\documentclass[10pt]{article}

\usepackage[ruled,linesnumbered]{algorithm2e}
\usepackage{amscd}
\usepackage{times}
\usepackage{amsmath}
\usepackage{amssymb}
\usepackage{amsthm}
\usepackage{dsfont}
\usepackage{mathrsfs}
\usepackage{relsize}
\usepackage{graphicx,graphics}
\usepackage{pstricks}
\usepackage{pst-all}
\usepackage[body={15cm,23cm}, top=2cm]{geometry}
\usepackage{bm}
\usepackage{multirow}
\usepackage{blkarray}
\usepackage{setspace}
\usepackage{enumerate}
\usepackage[inline,shortlabels]{enumitem}
\usepackage{cases}
\usepackage[linewidth=1pt]{mdframed}
\usepackage[numbers,sort&compress]{natbib}
\usepackage[e]{esvect}
\usepackage{tablefootnote}
\usepackage[colorlinks, citecolor=blue]{hyperref}
\hypersetup{linkcolor=blue,
anchorcolor=blue,
citecolor=blue }

\graphicspath{{fig/}}
\allowdisplaybreaks
\newtheorem{theorem}{Theorem}[section]
\newtheorem{corollary}[theorem]{Corollary}

\newtheorem{definition}[theorem]{Definition}
\newtheorem{remark}[theorem]{Remark}

\newtheorem{lemma}[theorem]{Lemma}

\def \st {\mathrm{ s.t. }}

\def \sign {\mathrm{sgn}}

\def \sta {\mathrm{star}}
\def \supp {\mathrm{supp}}
\def \R {\mathbb{R}}

\begin{document}

\title {Nonconvex models for recovering images
corrupted by salt-and-pepper noise on surfaces}

\author{Yuan Liu \thanks{School of Mathematics, Hefei University 
of Technology, Hefei, China. E-mail: liuyzz@hfut.edu.cn.
Y. Liu's research is supported in part by the National Natural Science
Foundation of China (No. 62002095).}
\and
Peiqi Yu \thanks{School of Mathematical Sciences and LPMC, Nankai University, 
Tianjin, China. E-mail: 2120220074@mail.nankai.edu.cn.}
\and
Chao Zeng  \thanks{Corresponding author. School of Mathematical Sciences 
and LPMC, Nankai University, Tianjin, China. E-mail: zengchao@nankai.edu.cn.
C. Zeng's research is supported in part by the National Natural Science 
Foundation of China (No. 12201319) and the Fundamental Research Funds for 
the Central Universities, Nankai University (No. 63231142).} }

\date{}
\maketitle

\emph{\textbf{Abstract\ } Image processing on surfaces has drawn 
significant interest in recent years, particularly in the context of denoising. 
Salt-and-pepper noise is a special type of noise which randomly sets a 
portion of the image pixels to the minimum or maximum intensity while 
keeping the others unaffected. In this paper, We propose the L$_p$TV models 
on triangle meshes to recover images corrupted by salt-and-pepper noise
on surfaces. We establish a lower bound for data fitting term of 
the recovered image. Motivated by the lower bound property, we propose the 
corresponding algorithm based on the proximal linearization method with 
the support shrinking strategy. The global convergence of the proposed 
algorithm is demonstrated. Numerical examples are given to show good 
performance of the algorithm.
}

\vspace{0.5cm}
\noindent \textbf{Keywords\ }  image on surface, image denoising, 
salt-and-pepper noise, variational methods, nonconvex optimization

\vspace{0.5cm}
\noindent \textbf{Mathematics Subject Classifications (2010)\ }

\section{Introduction}

Image processing on surface is widely used in computer vision
\cite{John2016Image}, computer graphics\cite{Bajaj2003Anisotropic}, and 
medical imaging \cite{Shi2013Cortical}. To generate an image on 
a surface, we need to transform an image defined on a surface into 
a point cloud by sampling specific points, each associated with a pixel 
value. Subsequently, a triangle mesh is constructed from this point cloud, 
where each vertex corresponds to a pixel value, as illustrated in 
Figure \ref{meshandimage}. The pixel values at arbitrary points on the 
triangle mesh are estimated by interpolation. There is still a relative 
scarcity of research in image processing on surface. Current work includes 
methods based on partial differential equations 
\cite{benninghoff2016segmentation,wu2009scale} and, more commonly, variational 
approaches. For example, variational methods based image segmentation is 
discussed in \cite{delaunoy2009convex,huska2018convex} and image decomposition 
is considered in \cite{wu2013variational,huska2019convex}.

When the data are gathered, errors may be introduced by measurement 
tools, processing or by experts. These errors are called noise. For an 
image on surface, errors may be introduced when sampling. That is, the pixel 
values at the vertices of a triangle mesh may be corrupted by noise.   
Denoising is a fundamental challenge in image processing. Various denoising 
techniques have been proposed in the literature. In the context of 2D images, 
among the methods employed for image denoising are filtering techniques, 
PDEs-based methods, wavelet-based approaches, and variational methods. 
Variational methods, in particular, focus on recovering 2D images by solving 
optimization models that promote gradient sparsity. In recent years, nonconvex 
and nonsmooth variational methods have gained prominence 
\cite{bian2015linearly,hintermuller2013nonconvex,nikolova2008efficient,
ochs2015iteratively} due to their remarkable advantages over convex and smooth 
variational methods, particularly in edge-preserving image restoration. These 
methods are especially effective in restoring high-quality, piecewise constant 
images with neat edges \cite{nikolova2005analysis,zeng2018edge}. Despite their 
advantages, designing efficient algorithms for nonconvex models remains a 
challenging task. Surface image denoising presents significantly greater 
challenges compared to 2D image denoising, as effective denoising in this 
context requires careful consideration of both the data set and the underlying 
mesh structure. Researches on image denoising techniques for triangle meshes 
are few in the existing literature, with only a few studies addressing this 
topic \cite{wu2008diffusion,wu2009scale,lai2011framework,wu2012augmented,
herrmann2018analysis,liu2024non}. Among these, the works of 
\cite{wu2008diffusion,wu2009scale} employ PDE-based methods, with 
\cite{wu2009scale} focusing on a full scale-space analysis. In contrast, 
studies \cite{lai2011framework,wu2012augmented,herrmann2018analysis,liu2024non} 
adopt variational methods, where an analogy to 
the total variation (TV) model \cite{rudin1992nonlinear} on surfaces is 
proposed and examined. A central challenge in these approaches is the 
definition of a discrete gradient for data sets on triangle meshes. A 
commonly adopted strategy is to interpolate the discrete data set over 
the mesh, thereby enabling the definition of the discrete gradient through 
the interpolation function. Linear interpolation is predominantly used in 
studies \cite{lai2011framework,wu2012augmented,wu2008diffusion,wu2009scale,
liu2024non}, while \cite{herrmann2018analysis} 
employs Raviart-Thomas basis functions for interpolation.

\begin{figure}[!ht]
\begin{center}
\begin{tabular}{cccc}
\includegraphics[height=3.0cm]{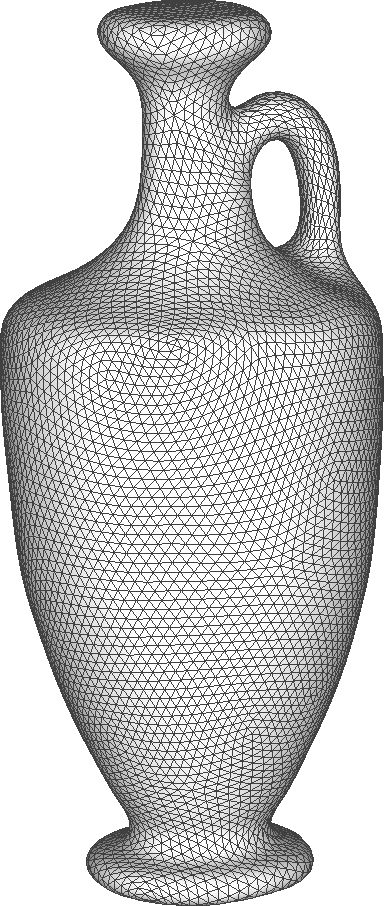}  &
\includegraphics[height=3.0cm]{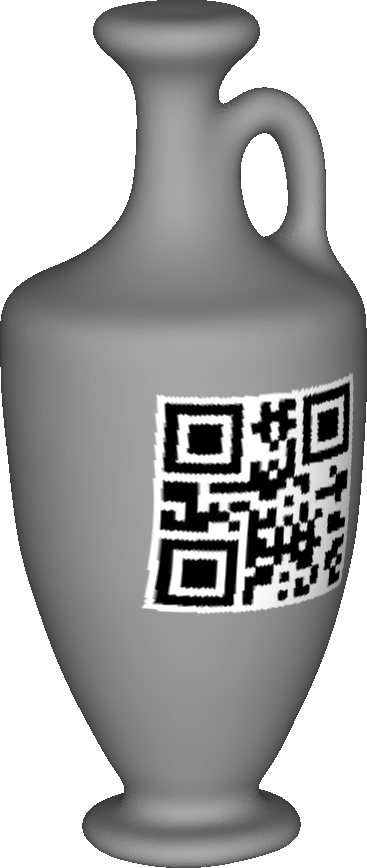}  &
\includegraphics[height=2.6cm]{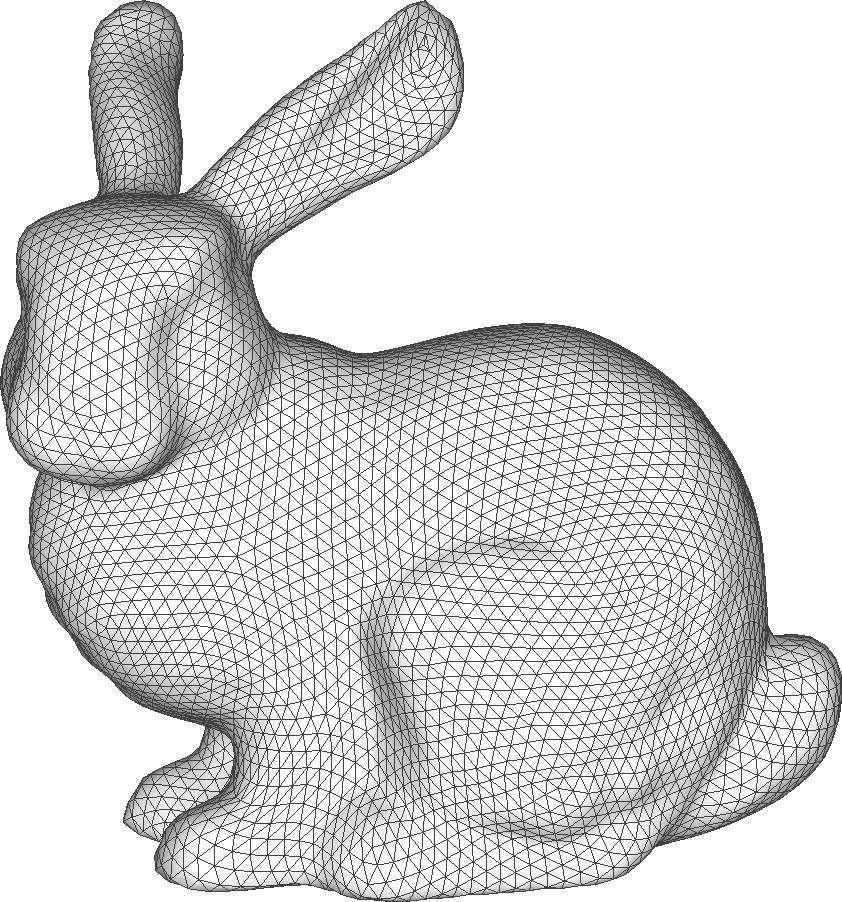} 
\includegraphics[height=2.6cm]{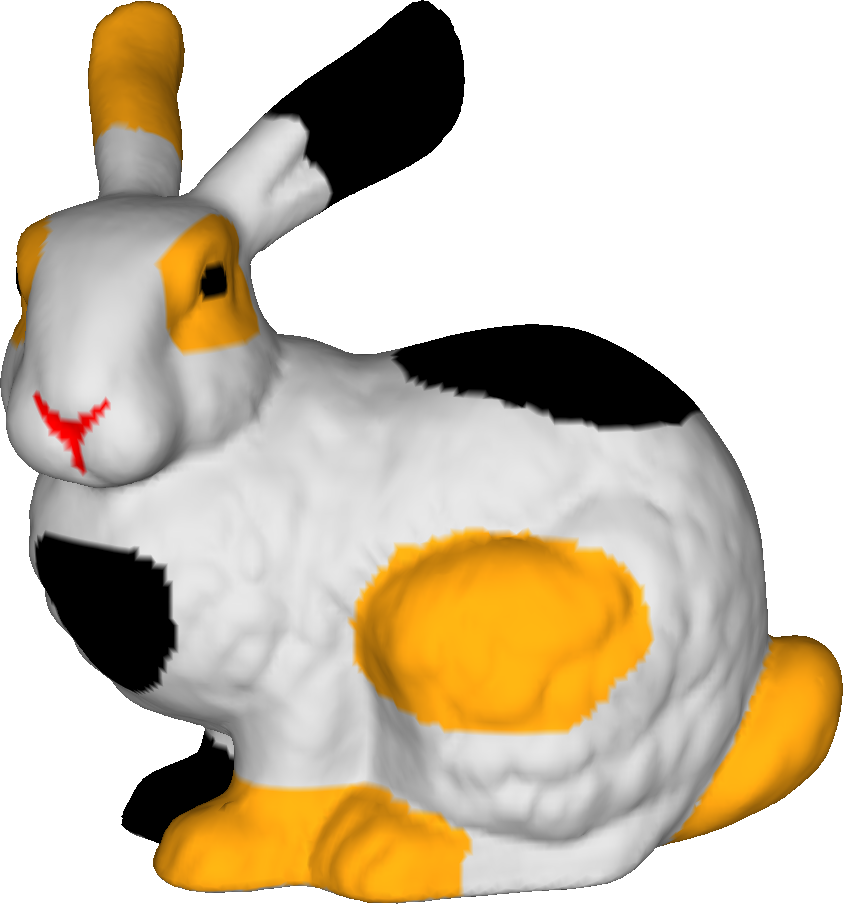} 
\end{tabular}
\caption{Triangle meshes and images on triangle meshes.}
\label{meshandimage}
\end{center}
\end{figure}

Unlike many papers focusing on gaussian noise, in this paper, we consider 
recovering images on surface corrupted by salt-and-pepper noise. This kind 
of noise is usually caused by sudden disturbances in image acquisition, 
such as errors in data transmission, faulty sensors, or analog-to-digital 
conversion errors. In an image with salt-and-pepper noise, some pixels 
are randomly set to the minimum (black) or maximum (white) intensity, while 
others remain unaffected. For 8 bits images, all pixels of the original image 
$\underline{u}$ range from 0 to 255, i.e., $0 \leq \underline{u}_j \leq 255$, 
for all $j$. Then, for the observed data $f$ corrupted by salt-and-pepper 
noise with noise level $\rho$ , we have
\begin{equation}\label{salt}
|\underline{u}_j-f_j|=\begin{cases}
|0-\underline{u}_j|, & \mbox{with probability } \frac{\rho}{2}; \\
|255-\underline{u}_j|, & \mbox{with probability } \frac{\rho}{2}; \\
0, & \mbox{with probability } 1-\rho.
\end{cases}
\end{equation}
It is clear that the restoration of $\underline{u}$ from the data 
corrupted by salt-and-pepper noise has the demand of sparse reconstruction 
of the data fitting term: $(1-\rho)$ percent of entries of the data fitting 
term with the true solution $\underline{u}$ are zero for noise level $\rho$.

From the perspective of sparse reconstruction, $\ell_0$ data fitting term is 
a natural choice, but it is an NP-hard problem. Instead, in 2D image 
processing, the most general alternative is the $\ell_1$
data fitting term \cite{Nikolova2013On,yang2009efficient,chan2005aspects}.
The $\ell_p$ minimization has advantages over the $\ell_1$ minimization
in sparse reconstruction; see \cite{candes2008enhancing,lai2013improved,
nikolova2005analysis,zeng2018edge,liu2024non}. Hence, it is reasonable to 
consider using the $\ell_p$ minimization for the data fitting term. 
Consequently, we consider image denoising on triangle meshes with nonconvex 
variational method. We utilize linear interpolation like \cite{lai2011framework}. 
We propose the L$_p$TV model on triangle meshes. Based on the corresponding 
lower bound property of the model, we adopt the proximal linearization method
(PLM). Every subproblem can be solved by the alternating direction method of 
multipliers (ADMM). We obtain the global convergence of the PLM. The
efficiency of the propose method is illustrated by related numerical 
experiments.

The rest of this paper is organized as follows. In Section \ref{sec2}, we 
introduce basic definitions for images on triangle meshes. In 
Section \ref{sec3}, we present the models, establish a lower bound for 
the data fitting term of local minimizers, and propose the corresponding 
algorithm. The convergence 
analysis are given in Section \ref{sec4}. Applications are given in 
Section \ref{sec5}. Conclusions are presented in Section \ref{sec6}.

\subsection*{Notation}
For a real-valued matrix $G$, $G^\top$ is its transpose, and $\|G\|$ is its 
spectral norm. For two vectors $g,h \in \mathbb{R}^{r}$, the inner product of 
the two vectors is $\langle g,h\rangle = g^\top h$, and the $\ell_2$ norm of 
$g$ is $\|g\|$. Denote by $e_j$ the column vector whose $j$th component is 
one while all the others are zeros. For a real number $r$,
the signum function $\sign(r)$ is defined as
$$
\sign(r)=\begin{cases}
           -1, & \mbox{if } r<0; \\
           0, & \mbox{if } r=0; \\
           1, & \mbox{if } r>0.
         \end{cases}
$$
The exact notation and definition of subdifferential can be found 
in Appendix \ref{apsd}

\section{Preliminaries}\label{sec2}

We introduce the discretization of surfaces and images on surfaces.
These techniques can be found in \cite{wu2009scale,
hirani2003discrete,huska2019convex}.

\subsection{Triangulated surfaces}

A connected surface in $\R^3$ can be approximated by a triangle mesh.
Assume that $\triangle$ is a triangle mesh in $\R^3$ with no degenerate 
triangles. The set of the indices of all vertices and the set of the 
indices of all triangles of $\triangle$ are denoted by $I_{v}$ and 
$I_{\tau}$, respectively. Then, the set of vertices and the set of triangles 
of $\triangle$ are denoted as $\{v_i:i \in I_v\}$, and 
$\{\tau_i:i \in I_{\tau}\}$, respectively. The numbers of vertices and 
triangles are denoted by $\mathrm{N}_v$ and $\mathrm{N}_{\tau}$,
respectively. A triangle whose vertices are $v_i,v_j,v_k$ is denoted as 
$[v_i,v_j,v_k]$. The area of a triangle $\tau$ is denoted by $|\tau|$.

For a vertex $v$, we define the \emph{star} of $v$, denoted by $\sta(v)$, to be 
the set of all triangles in $\triangle$ which share the vertex $v$. For the mesh 
$\triangle$, the \emph{dual mesh} is formed by connecting the barycenter and the 
middle point of each edge in each triangle. As illustrated in Figure
\ref{fig-DMeshCcells}, the original mesh $\triangle$ consists of black lines, 
while the dual mesh is in red. The \emph{control cell} of a vertex $v_i$ is part 
of its star which is near to $v_i$ in the dual mesh. Figure \ref{fig-DMeshCcells} 
shows the control cell $C_i$ for an interior vertex $v_i$ of the original mesh 
and the control cell $C_j$ for a boundary vertex $v_j$. See \cite{wu2009scale,
hirani2003discrete} for more details. 

\begin{figure}[!ht]
\begin{center}
\small
\begin{tabular}{ccc}
\includegraphics[width=2.8cm]{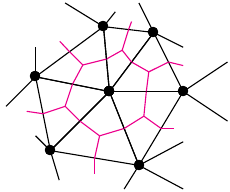} &
\includegraphics[width=2.3cm]{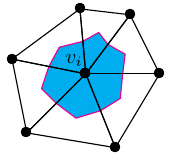} &
\includegraphics[width=2.3cm]{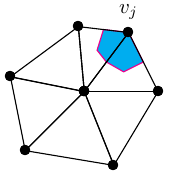}   \\
a mesh and its dual mesh  &
control cell of an interior vertex    &
control cell of a boundary vertex
\end{tabular}
\caption{\label{fig-DMeshCcells} Dual mesh and control cells.}
\end{center}
\end{figure}

\subsection{Discretization of images on surfaces}\label{sec_tri}

A grayscale image on the triangulated surface $\triangle$ is 
a function $\mathfrak{u} : \triangle \rightarrow \mathbb{R}$
with $\mathfrak{u}(x)$ being the pixel value of the point $x$ 
on $\triangle$. The image $\mathfrak{u}$ is sampled over the
vertices of the mesh and the value $\mathfrak{u}(x)$ at any point 
$x$ on $\triangle$ is estimated by linear interpolation. That is,
$\mathfrak{u}$ is understood as a piecewise linear function.

Let 
$$
u_i := \mathfrak{u}(v_i), \quad i=1,2,\ldots,\mathrm{N}_v,
$$
and let $u=[u_1,u_2,\ldots,u_{\mathrm{N}_v}]^\top\in 
\mathbb{R}^{\mathrm{N}_v}$. Denote the piecewise linear space by
$$
\mathcal{L}(\triangle):=\{\mathfrak{u}\in C^0(\triangle): 
\mathfrak{u}|_{\tau_i} \in \mathbb{P}_1,  i \in I_{\tau}\},
$$
where $\mathbb{P}_1$ denotes the space of trivariate linear polynomials. 
For each vertex $v_i$, we use $\phi_i$ to denote the piecewise linear 
function with support on $\sta(v_i)$, i.e., 
$$
\phi_i \in \mathcal{L}(\triangle) \mbox{ and }
\phi_i(v_j)=\delta_{ij}, \,i,j \in I_v,
$$
where $\delta_{ij}$ is the Kronecker delta. Then
$\{\phi_i: i=1,\ldots,\mathrm{N}_v\}$ forms a basis of 
$\mathcal{L}(\triangle)$. The function obtained by 
linear interpolation is
\begin{equation}\label{imageu}
\mathfrak{u}(x)=\sum\limits_{ i \in I_v}u_i\phi_i(x).
\end{equation}
Since $\mathfrak{u}$ is determined by $u$, for convenience,
we also call $u$ an image on surface.

By \cite[(4.2)]{wu2009scale}, we have
$$
\int_{\triangle} \mathfrak{u}\mathrm{\ d}x = \sum_{\tau=[v_l,v_m,v_n]} 
\frac{|\tau|}{3}(u_l+u_m+u_n)
= \sum_{i \in I_v} u_i
\sum_{\tau \in \sta (v_i)} \frac{|\tau|}{3} = \sum_{i \in I_v} u_i s_i,
$$
where $s_i$ is the area of the control cell of the vertex $v_i$.
The discrete gradient of $u$ can be defined by the gradient of $\mathfrak{u}$. 
Since $\mathfrak{u}$ is only $C^0$ on $\triangle$, we consider the gradient 
of $\mathfrak{u}$ on every triangle
$$
\nabla \mathfrak{u}|_{\tau} = \sum_{i \in I_v} u_i \nabla \phi_i|_{\tau}.
$$
The calculation of $\nabla \phi_i|_{\tau}$ is given in \cite{wu2009scale} 
by using the natural piecewise parametrization. See Figure \ref{figgra}, 
where $O$ is the Euclidean projection of $v_i$ on the line passing through 
$v_j,v_k$.

\begin{figure}[!ht]
\begin{center}
\includegraphics[width=3.5cm]{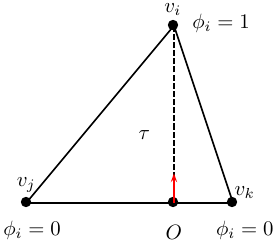}
\caption{\label{figgra} The gradient of the basis $\phi_i$ restricted on 
$[v_i,v_j,v_k]$: the red vector.}
\end{center}
\end{figure}

Denoting $h_i^{\tau}=v_i-O$ (see the red vector in Figure \ref{figgra}), 
the deduction in \cite{wu2009scale} shows that the gradient of $\phi_i$ 
restricted on $\tau=[v_i,v_j,v_k]$ is
$$
\nabla \phi_i|_{[v_i,v_j,v_k]}=\frac{h_i^{\tau}}{|h_i^{\tau}|^2}.
$$
As one can see, $\nabla \phi_i|_{\tau}$ is a constant vector, and thus so is 
$\nabla \mathfrak{u}|_{\tau}$:
\begin{equation}\label{BasisonMesh}
\nabla \mathfrak{u}|_{\tau=[v_i,v_j,v_k]} = \frac{u_ih_i^{\tau}}{|h_i^{\tau}|^2}
+\frac{u_jh_j^{\tau}}{|h_j^{\tau}|^2}+\frac{u_kh_k^{\tau}}{|h_k^{\tau}|^2}.
\end{equation}
The discrete gradient operator is defined as
\begin{align*}
D : \quad \mathbb{R}^{\mathrm{N}_v} & \longrightarrow  
\mathbb{R}^{3\times\mathrm{N}_{\tau}} \\
u  & \longmapsto   \left(\nabla \mathfrak{u}|_{\tau_1},\ldots,
\nabla \mathfrak{u}|_{\tau_{\mathrm{N}_\tau}}\right).
\end{align*}
We denote by $D_i u$ the $i$th column of $Du$, i.e., 
\begin{equation}\label{D}
D_i \in \mathbb{R}^{3 \times \mathrm{N}_v}, \quad 
D_iu = \nabla \mathfrak{u}|_{\tau_i}.
\end{equation}
Suppose $\tau_i=[v_{i_1},v_{i_2},v_{i_3}]$. By
(\ref{BasisonMesh}), the $i_1$th column, the $i_2$th column and 
the $i_3$th column of $D_i$ are $\frac{h_{i_1}^{\tau_i}}{|h_{i_1}^{\tau}|^2}$, 
$\frac{h_{i_2}^{\tau_i}}{|h_{i_2}^{\tau}|^2}$,
$\frac{h_{i_3}^{\tau_i}}{|h_{i_3}^{\tau}|^2}$, respectively, and the other 
entries of $D_i$ are zero.

\section{Models and algorithm}\label{sec3}

We consider the restoration of an image corrupted by impulse noise:
$$
f = \underline{u} + n,
$$
where $f$ is the observed image, $\underline{u}$ is the original 
image, and $n$ is the impulse noise. The impulse noise corrupts
a portion of elements of $\underline{u}$ while keeping the other
elements unaffected. That is, many elements of $\underline{u}-f$
are zero. Suppose $\underline{\mathfrak{u}}$ is the linear interpolation 
function of $\underline{u}$ defined in (\ref{imageu}). Like 2D images,
we know that the norms of the gradient $|\nabla \underline{\mathfrak{u}}|$ 
are sparse. Therefore, we can use the idea of sparse reconstruction
to restore the original image.

We propose the following L$_p$TV model:
\begin{equation}\label{clptv}
\min_{\mathfrak{u} \in \mathcal{L}(\triangle)} \lambda 
\sum_{j \in I_{v}} |u_j-f_j|^p +
\int_{\triangle} |\nabla \mathfrak{u}| \mathrm{\ d}x,
\end{equation}
where $\lambda >0$ is a parameter and $0<p<1$.
With the symbols introduced in Section \ref{sec_tri}, the discrete 
version of (\ref{clptv}) is
\begin{equation}\label{tvp}
\min_{u\in \mathbb{R}^{\mathrm{N}_v}} \mathcal{F}(u):=
\lambda \sum_{j \in I_{v}} |u_j-f_j|^p +
\sum_{i \in I_{\tau}}|\tau_i| \|D_iu\|,
\end{equation}
where $D_i$ is defined in (\ref{D}). 
Obviously, $\mathcal{F}(u)$ is coercive. 
Hence, the solution of (\ref{tvp}) always exists.

\subsection{Lower bound theory}

\begin{definition}\label{sup}
For any $u \in\mathbb{R}^{\mathrm{N}_v}$, we define the support 
of the data fitting term as 
$$
{\rm supp}(u)=\{j \in I_{v} :u_{j}-f_{j} \neq 0\}.
$$
\end{definition}

We have the following lower bound theory for L$_p$TV model on triangle meshes.

\begin{theorem}\label{lowerbound}
For any local minimizer $u^{*}$ of $\mathcal{F}(u)$ in (\ref{tvp}),
there exists a constant $\theta > 0$ such that for all $j \in I_v$,
if $|u_j-f_j|\neq 0$, then
$$
|u^*_j-f_j| > \theta.
$$
\end{theorem}
\begin{proof}
For the local minimizer $u^*$ of $\mathcal{F}(u)$, define the 
following function
$$
\widehat{\mathcal{F}}(u):=
\lambda \sum_{j \in {\rm supp}(u^*)} |u_j-f_j|^p +
\sum_{i \in I_{\tau}}|\tau_i| \|D_iu\|,  
$$
and the affine set
$$
C:=\{v \in  \mathbb{R}^{\mathrm{N}_v}:v_j=f_j 
\enspace \forall j \in I_v\backslash {\rm supp}(u^*)\}.
$$
It is easy to find that
$$
u \in C \quad \Rightarrow \quad \widehat{\mathcal{F}}(u)=\mathcal{F}(u).
$$
Because $u^*$ is a local minimizer of $\mathcal{F}(u)$, 
$u^*$ is a local minimizer of $ \widehat{\mathcal{F}}(u)$ over $C$. 
The corresponding tangent cone of $C$ is
$$
T_C:=\{v \in  \mathbb{R}^{\mathrm{N}_v}:v_j=0 \enspace 
\forall j \in I_v\backslash {\rm supp}(u^*)\}.
$$
By noticing that $T_C$ is a linear space, the first order optimality 
condition tells us that there exists 
$\hat{v} \in \partial \widehat{\mathcal{F}}(u^*)$ such that
\begin{equation}\label{optc}
\langle \hat{v} ,u \rangle=0 \quad \forall u \in T_C.
\end{equation}

By \cite[Exercise 8.8]{rockafellar2009variational} and
\cite[Corollary 10.9]{rockafellar2009variational}, we have 
\begin{align*}
\partial\widehat{\mathcal{F}}(u^*)=&\, \lambda p 
\sum_{j \in {\rm supp}(u^*)}  |u_j^*-f_j|^{p-1}
\sign(u_j^*-f_j)e_j + \partial \left(
\sum_{i \in I_{\tau}}|\tau_i| \|D_iu^*\|\right) \\
=&\, \lambda p \sum_{j \in {\rm supp}(u^*)}  
|u_j^*-f_j|^{p-1}\sign(u_j^*-f_j)e_j + \sum_{i \in I_{\tau}}\partial \big(
|\tau_i| \|D_iu^*\|\big) \\
\subseteq&\, \lambda p \sum_{j \in {\rm supp}(u^*)}  
|u_j^*-f_j|^{p-1}\sign(u_j^*-f_j)e_j + \sum_{i \in I_{\tau}} 
\big\{|\tau_i| D_i^\top w_i: \|w_i\| \leq 1, w_i \in \mathbb{R}^3 \big\}. 
\end{align*}
By (\ref{optc}), there exists $\bar{w}_i \in \mathbb{R}^3$ 
satisfying $ \|\bar{w}_i\| \leq 1 \, \forall i \in I_{\tau} $ 
such that for any $ u \in T_C$, we have
\begin{equation}\label{proleq}
\lambda p \sum_{j \in {\rm supp}(u^*)}  |u_j^*-f_j|^{p-1}
\sign(u_j^*-f_j)u_j = - \sum_{i \in I_{\tau}} |\tau_i|\langle 
\bar{w}_i,D_iu \rangle \, \leq \sum_{i \in I_{\tau}} |\tau_i|\|D_i\|\|u\|.
\end{equation}
For every $j \in {\rm supp}(u^*)$, let $ v^j \in \mathbb{R}^{\mathrm{N}_v}$
be the solution to the problem
\begin{equation}\label{prob}
\begin{split}
\min_{v\in \mathbb{R}^{\mathrm{N}_v}} \quad & \|v\| \\
\st \quad & v \in T_C ,\\
          & v_j=\sign(u_j^*-f_j), \\
          & \sign(u_j^*-f_k)v_k \geqslant 0 \quad \forall 
          k \in {\rm supp}(u^*)\backslash\{j\}.
\end{split}  
\end{equation}
Since the objective function represents the distance between the 
origin and the feasible set, and noticing that the feasible set is 
closed, it follows that (\ref{prob}) 
is solvable. Define
$$
\mu=\max_{j\in {\rm supp}(u^*)} \{\|v^j\|\} \quad 
$$
For a fixed $\bar{j} \in {\rm supp}(u^*)$, letting 
$u= v^{\bar{j}}$ that satisfies (\ref{prob}) and substituting 
it into (\ref{proleq}), we have
\begin{eqnarray*}
\lambda p 
|u_{\bar{j}}^*-f_{\bar{j}}|^{p-1} &
=&\lambda p 
|u_{\bar{j}}^*-f_{\bar{j}}|^{p-1}
\sign(u_{\bar{j}}^*-f_{\bar{j}})v_{\bar{j}}^{\bar{j}} \\
&\leq&\lambda p \sum_{j \in {\rm supp}(u^*)}
|u_j^*-f_j|^{p-1}
\sign(u_j^*-f_j)v_j^{\bar{j}} \\
&\overset{(\ref{proleq})}{\leq}& \sum_{i \in I_{\tau}} 
|\tau_i|\|D_i\|\|v^{\bar{j}}\| \leq \mu\sum_{i \in I_{\tau}}
|\tau_i|\|D_i\|.
\end{eqnarray*}
So for every $j \in {\rm supp}(u^*)$, we have
$$
|u_j^*-f_j|^{p-1} \leq \dfrac{\mu\sum_{i \in I_{\tau}} 
|\tau_i|\|D_i\|}{\lambda p }.
$$
Define
\begin{equation}
\theta:=\left(\dfrac{\mu\sum_{i \in I_{\tau}} |\tau_i|
\|D_i\|}{\lambda p  }\right)^{\frac{1}{p-1}}.
\end{equation}
Then, $|u^*_j-f_j| > \theta$ for every $j \in {\rm supp}(u^*)$.
\end{proof}

\begin{corollary}\label{corlow}
Suppose $u^{*}$ is a local minimizer of $\mathcal{F}(u)$ in (\ref{tvp}). 
Let $\theta$ be the lower bound defined in Theorem \ref{lowerbound}. 
Suppose $\hat{u}$ is a point close to $u^*$ satisfying 
$|u_j^*-\hat{u}_j|<\theta \quad \forall j \in I_v$. If $|\hat{u}_j-f_j|=0$  
for some $j \in I_v$, then $ |u_j^*-f_j|=0$.
\end{corollary}
\begin{proof}
If $|\hat{u}_j-f_j|=0$, by the triangle inequality, we have
$$
|u_j^*-f_j| \leq |u_j^*-\hat{u}_j|+|\hat{u}_j-f_j| < \theta.
$$
Then by Theorem \ref{lowerbound}, we have $|u_j^*-f_j|=0$.
\end{proof}

\subsection{Algorithm}

In this section, we propose an algorithm based on the theories given 
in the previous section. Suppose we use an iterative algorithm to find 
a minimizer $u^*$ of (\ref{tvp}) and the iterate sequence is $u^k$.
By Corollary \ref{corlow}, if $\lim_{k\rightarrow \infty} u^k=u^*$,
there is a $k_0$ such that
$$
\mbox{ if } \, |u_j^k-f_j|=0 \mbox{ for some } k \geq k_0 
\mbox{ and some } j \in I_{v} \Rightarrow |u_j^*-f_j|=0.
$$
Motivated by this observation, we consider solving $u^{k+1}$ by
\begin{equation}\label{equ_alg}
\left\{
\begin{aligned}
\min_{u \in \mathbb{R}^{\mathrm{N}_v}} \quad & \lambda 
\sum_{j \in {\rm supp}(u^k)} |u_j-f_j|^p + 
\sum_{i \in I_{\tau}}|\tau_i| \|D_iu\|, \\
\st \quad &  u_j=f_j  \quad \forall j \in I_{v} 
\setminus \sup(u^k),
\end{aligned} \right.
\end{equation}
where $\sup(\cdot)$ is defined in Definition \ref{sup}. By considering
the reality of computation, we treat $|u_j^k-f_j|$ as zero if it is small
enough. Specifically, we define 
\begin{equation*}\label{supc}
\sup^{\varepsilon}(u) = \left\{j \in I_{v}: |u_j^k-f_j| > 
\varepsilon \right\},
\end{equation*}
and
$$
\Gamma^k=\sup^{\varepsilon}(u^k), \quad \Gamma_c^k=I_v \setminus 
\sup^{\varepsilon}(u^k),
$$
and change (\ref{equ_alg}) into 
\begin{equation}\label{equ_alg1}
\left\{
\begin{aligned}
\min_{u \in \mathbb{R}^{\mathrm{N}_v}} \quad & 
\mathcal{F}_k(u) := \lambda \sum_{j \in \Gamma^k}  |u_j-f_j|^p 
+ \sum_{i \in I_{\tau}}|\tau_i| \|D_iu\|, \\
\st \quad &  u_j=f_j  \quad \forall j \in \Gamma_c^k.
\end{aligned} \right.
\end{equation} 
Directly solving (\ref{equ_alg1}) is difficult. The linearization 
is a common technique for nonlinear optimization. That is, the 
first-order Taylor approximation of some term around $u^k$, is 
employed to obtain an approximate objective function which is easy to 
solve. Specifically, for all $j \in \Gamma^k$, we have
\begin{equation}\label{linear-lp}
|u_j-f_j|^p \approx |u_j^k-f_j|^p + p|u_j^k-f_j|^{p-1}
\left(|u_j-f_j|- |u_j^k-f_j| \right).
\end{equation}
For convenience, we define
\begin{equation}\label{sigwei}
w^k_j = p|u_j^k-f_j|^{p-1} \quad \forall j \in \Gamma^k.
\end{equation}
Since the function $x \mapsto x^p$ is concave on $[0,+\infty)$,
it follows that for all $j \in \Gamma^k$,
\begin{equation}\label{linear-cave}
|u_j-f_j|^p\leq w^k_j|u_j-f_j|+(|u_j^k-f_j|^p - w^k_j|u_j^k-f_j|).
\end{equation}
Substituting (\ref{linear-lp}) into (\ref{equ_alg1}), deleting some 
constants and adding a proximal term, we obtain the following
proximal linearization version of (\ref{equ_alg1}):
\begin{equation}\label{equ_lialg2}
\left\{
\begin{aligned}
\min_{u \in \mathbb{R}^{\mathrm{N}_v}} \quad & 
\mathcal{H}_k(u) := & & \lambda \sum_{j \in \Gamma^k} 
w_j^k|u_j-f_j| + \sum_{i \in I_{\tau}}|\tau_i| \|D_iu\| 
+\frac{\rho}{2}\|u-u^k\|^2, \\
\st \quad & u_j=f_j   & & \forall j \in \Gamma_c^k.
\end{aligned} \right.
\end{equation}
The proximal term $\frac{\rho}{2}\|u - u^k\|^2$ is to increase 
the robustness of the nonconvex algorithm and is helpful in the 
global convergence. By (\ref{linear-cave}), we can set $\rho$ to 
be any positive scalar.

The whole procedure of the \emph{proximal linearization method}
(PLM) is summarized in Algorithm \ref{alg}. Algorithm \ref{alg} can 
be extended to multichannel (e.g., color) image denoising.

\begin{algorithm}[H]
\caption{proximal linearization method (PLM) for 
solving (\ref{tvp})}\label{alg}
\small
\SetAlgoLined
\DontPrintSemicolon
\SetKw{End}{end function}
\SetKwInput{Ini}{Initialize}
\KwIn{$\triangle,\lambda,\varepsilon,f$ }
\Ini{$u^0, k=0$}
\Repeat{}{
Update $\Gamma^k,\Gamma^k_c$ \;
Update $w^{k+1}_j, j \in \Gamma^k$ by (\ref{sigwei}) \;
Solve $u^{k+1}$ by (\ref{equ_lialg2}) \;
$k\leftarrow k+1$
}( termination criteria met)
\end{algorithm}

Next, we introduce how to solve (\ref{equ_lialg2}) by the alternating
direction method of multipliers (ADMM).

\subsubsection{Algorithm implementation}

We reformulate $\mathcal{H}_k(u)$ in (\ref{equ_lialg2}). The variable 
$u$ is replaced by $v$ to avoid confusion with the outer loops. We 
introduce two new variables: $y=(y_j)$, where $y_j \in \mathbb{R}$ 
and $j \in \Gamma^k$, and $z=(z_i)$, where $z_i \in \mathbb{R}^3$ and 
$i \in I_\tau$. Then (\ref{equ_lialg2}) is reformulated as
\begin{equation*}
\begin{split}
\min_{v,y,z} \quad & \lambda \sum_{j \in \Gamma^k} 
w_j^k|y_j| + \sum_{i \in I_{\tau}}|\tau_i| \|z_i\| 
+\frac{\rho}{2}\|v-u^k\|^2, \\
\st \quad & y_j=v_j-f_j \quad \forall j \in \Gamma^k ,\\
& z_i=D_iv \quad \forall i \in I_\tau, \\
& v_j=f_j \quad \forall j \in \Gamma_c^k.
\end{split}  
\end{equation*}
The augmented Lagrangian function for the above problem is defined as
\begin{align*}
\mathcal{L}(v,y,z;\eta,\mu) = & \lambda \sum_{j \in \Gamma^k} 
w_j^k|y_j| + \sum_{i \in I_{\tau}}|\tau_i| \|z_i\| 
+\frac{\rho}{2}\|v-u^k\|^2 \\
& + \sum_{j \in \Gamma^k}\eta_j(v_j-f_j-y_j)+ \sum_{j \in \Gamma_c^k}
\eta_j(v_j-f_j) + \sum_{i \in I_{\tau}}|\tau_i|\langle 
\mu_i,D_iv-z_i \rangle \\
& + \frac{\beta_1}{2}\sum_{j \in \Gamma^k} [y_j-(v_j-f_j)]^2+ 
 \frac{\beta_1}{2}\sum_{j \in \Gamma_c^k} (v_j-f_j)^2+ 
 \frac{\beta_2}{2}\sum_{i\in I_\tau}|\tau_i| \|z_i- D_i v\|^2,
\end{align*}
where $\beta_1,\beta_2 > 0, \eta_j \in \mathbb{R}, j \in I_v$ and 
$\mu_i \in \mathbb{R}^3, i \in I_\tau$ are Lagrange multipliers.
Applying the ADMM yields Algorithm \ref{admm}.

\begin{algorithm}[H]
\caption{ADMM for solving  (\ref{equ_lialg2})}\label{admm}
\small
\SetAlgoLined
\DontPrintSemicolon
\SetKw{End}{end function}
\SetKwInput{Ini}{Initialize}
\KwIn{$u^k,\lambda,\rho, \beta_1, \beta_2$ }
\Ini{$v^0=u^k, \eta^0=0, \mu^0=0, t=0$}
\Repeat{}{
Compute $(y^{t+1},z^{t+1})$ by solving
\begin{equation}\label{sub_yz}
(y^{t+1},z^{t+1}) \in \arg \min_{y,z}\mathcal{L}(v^{t},
y,z;\eta^t,\mu^t)
\end{equation}\;  \vspace*{-.4cm}
Compute $v^{t+1}$ by solving
\begin{equation}\label{sub_v}
v^{t+1} \in \arg \min_{v}\mathcal{L}(v,y^{t+1},z^{t+1};
\eta^t,\mu^t)
\end{equation}\;  \vspace*{-.4cm}
Update $\eta^{t+1}_j = \begin{cases}
\eta^{t}_j+ \beta_1 (v_j^{t+1}-f_j-y_j^{t+1})
\quad \forall j \in \Gamma^k    \\
\eta^{t}_j+ \beta_1 (v_j^{t+1}-f_j)\quad \forall 
j \in \Gamma_c^k    
\end{cases}$\;
Update $\mu_i^{t+1}=\mu_i^{t}+ \beta_2(D_iv^{t+1}-z_i^{t+1})
\quad \forall i \in I_\tau$ \;
$t\leftarrow t+1$
}( termination criteria met)
\end{algorithm}

The two subproblems in the above algorithm are calculated as follows.
\begin{enumerate}
\item The $(y,z)$-subproblem (\ref{sub_yz}): we can simplify 
(\ref{sub_yz}) as
\begin{equation*}
\begin{aligned}
(y^{t+1},z^{t+1}) \in \arg \min_{y,z} \biggl\{ &\sum_{j \in \Gamma^k}
\left(   \lambda w_j^k|y_j|+\dfrac{\beta_1}{2} \left|y_j- \left( v_j^t-f_j
+\dfrac{\eta_j^t}{\beta_1} \right) \right|^2 \right)    \\ 
+ & \sum_{i \in I_{\tau}}|\tau_i| \left(\|z_i\|+\dfrac{\beta_2}{2} 
\left\|z_i-\left(  D_iv^t+\dfrac{\mu_i^t}{\beta_2}\right)  
\right\|^2   \right)  \biggr\}.
\end{aligned}
\end{equation*}
The solution of above problem is given by the following one-dimensional 
and three-dimensional shinkage:
\begin{equation*}
\begin{aligned}
y_j^{t+1} =& \max \left\{\left|v^t_j-f_j+\frac{\eta^t_j}{\beta_1}\right|
-\frac{\lambda w^k_j}{\beta_1},0 \right\}\circ \sign\left(v^t_j-f_j+
\frac{\eta^t_j}{\beta_1}\right)
\quad \forall j \in \Gamma^k;   \\
z_i^{t+1} =& \max \left\{\left\|D_iv^t + \frac{\mu_i^t}{\beta_2}\right \|
-\frac{1}{\beta_2}, 0 \right\}\frac{D_iv^t
	+ \frac{\mu_i^t}{\beta_2}}{\big\|D_iv^t + \frac{\mu_i^t}{\beta_2} \big\|}
 \quad \forall i \in I_\tau,
\end{aligned}
\end{equation*}
where the convention $0 \cdot (0/0)=0$ is followed, and $\circ$  denotes 
the pointwise product.
\item  The $v$-subproblem (\ref{sub_v}): we introduce $\tilde{y}=(y_j)$ 
with $y_j=0 \, \forall j \in \Gamma_c^k$, and define 
$\overline{y}^{t+1}=(y^{t+1},\tilde{y}^{t+1})$. 
Then we can simplify (\ref{sub_v}) as
\begin{equation*}
\begin{aligned}
v^{t+1} \in \arg \min_{v} \bigg\{ & \frac{\rho}{2}\|v-u^k\|^2
+ \sum_{j \in I_v}\eta_j^t(v_j-f_j-\overline{y}_j^{t+1})+ 
\sum_{i \in I_{\tau}}|\tau_i|\langle \mu_i^t,D_iv-z_i^{t+1} \rangle \\
& + \frac{\beta_1}{2}\sum_{j \in I_v} [\overline{y}_j^{t+1}-(v_j-f_j)]^2
+ \frac{\beta_2}{2}\sum_{i\in I_\tau}|\tau_i| \|z_i^{t+1}- D_i v\|^2 \bigg\}.
\end{aligned}
\end{equation*}
This is a least squares problem and the corresponding normal equation is
\begin{equation*}
\left(\beta_2\sum_{i \in I_\tau}|\tau_i|D_i^\top D_i+ (\rho+\beta_1 )I \right)v
=\sum_{i \in I_\tau}|\tau_i|D_i^\top(\beta_2z_i^{t+1}-\mu_i^t)+
\beta_1(\overline{y}^{t+1}+f)-\eta^t+\rho u^k
\end{equation*}
This is a sparse linear system and can be solved by various 
well-developed numerical packages.

\end{enumerate}

\section{Convergence analysis}\label{sec4}

In this section, we establish the global convergence of $\{u^k\}$,
generated by Algorithm \ref{alg}. Our proof is based on the 
Kurdyka-{\L}ojasiewicz (KL) property of $\mathcal{F}(u)$. For examples 
of KL functions. For the reader's convenience, we provide
some related facts and results on KL functions in Appendix \ref{apKL}.

For the convergence analysis of nonconvex optimization, the KL 
property is pretty significant, which has drawn a great many of 
attention in recent years. For the development of the application 
of KL property in optimization theory, see \cite{bolte2007lojasiewicz,
bolte2007clarke,attouch2009convergence,	attouch2010proximal,
attouch2013convergence,bolte2014proximal} and references therein. 
One can refer to \cite{attouch2010proximal,attouch2013convergence,bolte2014proximal},
\cite[Proposition 6]{ochs2015iteratively}, \cite[section 3.3]{zeng2019iterative}
for examples of KL functions. For this paper, $\mathcal{F}(u)$ is a KL function.

First, by (\ref{equ_lialg2}), we have the following relationship:
\begin{equation} \label{subset}
\Gamma^{k+1} \subseteq {\rm supp}(u^{k+1}) \subseteq \Gamma^k 
\subseteq {\rm supp}(u^k).
\end{equation}
Hence, $\{\# \Gamma^k\}$ is a non-increasing sequence. Since $\# \Gamma^k 
\in \{0,1,\ldots,N_v\}$, it follows that $\{\Gamma^k\}$ will converge within 
finite steps, i.e., \begin{equation}\label{equ_unchangeconstaint}
\exists K, \ \mbox{ such that }\  \Gamma^k =\Gamma^K \quad \forall k \geq K.
\end{equation}
We denote $\Gamma = \Gamma^K$ and $\Gamma_c =I_{\tau} \setminus \Gamma^K$.
Then, when $k \geq K$, (\ref{equ_lialg2}) becomes:
\begin{equation}\label{equ_lialg3}
\left\{
\begin{aligned}
\min_{u \in \mathbb{R}^{\mathrm{N}_v}} \quad & 
\overline{\mathcal{H}}_k(u) := & & \lambda \sum_{j \in \Gamma} 
w_j^k|u_j-f_j| + \sum_{i \in I_{\tau}}|\tau_i| \|D_iu\| 
+\frac{\rho}{2}\|u-u^k\|^2, \\
\st \quad & u_j=f_j   & & \forall j \in \Gamma_c.
\end{aligned} \right.
\end{equation}
In order to eliminate the constraint of (\ref{equ_lialg3}) for the 
convenience of convergence analysis, we define
$$
\mathscr{C}:=\{u \in \mathbb{R}^{\mathrm{N}_v}: u_j=f_j  \quad 
\forall j \in \Gamma_c\},
$$
and the corresponding indicator function of $\mathscr{C}$ as 
$\delta_{\mathscr{C}}$, where $\delta_{\mathscr{C}}:
\mathbb{R}^{N_v} \rightarrow (-\infty,+\infty]$ is defined by
$$
\delta_{\mathscr{C}}(u)=\begin{cases}
0 & \mbox{if } u \in \mathscr{C}, \\
+\infty & \mbox{otherwise}.
\end{cases}
$$
Then
$$
\partial \big( \delta_{\mathscr{C}}(u) \big)=N_{\mathscr{C}}(u)=
\begin{cases}
\left\{\sum_{j \in \Gamma_c} b_je_j: b_j \in \mathbb{R}\right\}
& \mbox{if } u \in \mathscr{C}, \\
\emptyset & \mbox{otherwise},
\end{cases}
$$
where $N_{\mathscr{C}}(u)$ is the normal cone of $\mathscr{C}$ at $u$.
So (\ref{equ_lialg3}) is equivalent to the following problem:
\begin{equation}\label{pro_linlast}
\min_{u \in \mathbb{R}^{\mathrm{N}_v}} \quad 
\overline{\mathcal{H}}^{\delta}_k(u):= \overline{\mathcal{H}}_k(u)
+ \delta_{\mathscr{C}}(u).
\end{equation}
When $k \geq K$, $u^k \in \mathscr{C}$ and $u^{k+1}$ solves 
$\overline{\mathcal{H}}^{\delta}_k(u)$ in (\ref{pro_linlast}).

In order to enable readers to better understand the proof, we review 
the definition of $\mathcal{F}(u), \mathcal{F}_k(u)$ and $\mathcal{H}_k(u)$:
\begin{equation*}
\begin{aligned}
&\mathcal{F}(u):=
\lambda \sum_{j \in I_{v}} |u_j-f_j|^p +
\sum_{i \in I_{\tau}}|\tau_i| \|D_iu\|, \\
&\mathcal{F}_k(u) := \lambda \sum_{j \in \Gamma^k}  |u_j-f_j|^p 
+ \sum_{i \in I_{\tau}}|\tau_i| \|D_iu\|, \\
&\mathcal{H}_k(u) := \lambda \sum_{j \in \Gamma^k} 
w_j^k|u_j-f_j| + \sum_{i \in I_{\tau}}|\tau_i| \|D_iu\| 
+\frac{\rho}{2}\|u-u^k\|^2.
\end{aligned}
\end{equation*}
By the algorithm and (\ref{subset}), the following relation can be verified 
easily:
\begin{equation}\label{20}
\mathcal{F}_k(u^{k+j})=\mathcal{F}(u^{k+j}) \quad \forall k \geq 0, 
\forall j \geq 1.
\end{equation}

\begin{lemma}\label{lemma_decrease}
Let $\{u^k\}$ be generated by Algorithm \ref{alg}. Then, the objective 
function sequence $\{\mathcal{F}(u^k)\}$ satisfies 
\begin{equation}\label{equ_Fnonincrease}
\frac{\rho}{2}\|u^{k+1}-u^k\|^2 \leq \mathcal{F}(u^k) - 
\mathcal{F}(u^{k+1}) \quad \forall \, k \geq 0,
\end{equation}
and $\{u^k\}$ is bounded.
\end{lemma}

\begin{proof}
According to (\ref{linear-cave}), we have
\begin{equation}\label{22}
\mathcal{F}_k(u)+\frac{\rho}{2}\|u-u^k\|^2 \leq \mathcal{H}_k(u)+ 
\lambda \sum_{j \in \Gamma^k}(|u^k_j-f_j|^p-w_j^k|u^k_j-f_j|).
\end{equation}
Letting $u=u^k$ yields that
\begin{equation}\label{23}
\begin{aligned}
\mathcal{H}_k(u^k)+ \lambda \sum_{j \in \Gamma^k}(|u^k_j-f_j|^p
-w_j^k|u^k_j-f_j|)=& \lambda \sum_{j \in \Gamma^k}|u^k_j-f_j|^p
+ \sum_{i \in I_{\tau}}|\tau_i| \|D_iu^k\| \\
=& \mathcal{F}(u^k).
\end{aligned}
\end{equation}
It follows that
\begin{equation*}
\begin{aligned}
& \mathcal{F}(u^{k+1})+\frac{\rho}{2}\|u^{k+1}-u^k\|^2 \\
\overset{(\ref{20})}{=}& \mathcal{F}_k(u^{k+1})+\frac{\rho}{2}
\|u^{k+1}-u^k\|^2\\
\overset{(\ref{22})}{\leq}& \mathcal{H}_k(u^{k+1})+ \lambda 
\sum_{j \in \Gamma^k}(|u^k_j-f_j|^p-w_j^k|u^k_j-f_j|) \\
\leq&  \mathcal{H}_k(u^k)+ \lambda \sum_{j \in \Gamma^k}
(|u^k_j-f_j|^p-w_j^k|u^k_j-f_j|) \\
\overset{(\ref{23})}{=}& \mathcal{F}(u^k),
\end{aligned}
\end{equation*}
which proves (\ref{equ_Fnonincrease}).

Combining (\ref{equ_Fnonincrease}) and the fact that $\mathcal{F}(u^k) 
\geq 0$ yields the convergence of $\{\mathcal{F}(u^k) \}$. Since 
$\mathcal{F}(u)$ is coercive, it follows that $\{u^k\}$ is bounded.
\end{proof}

Next, we give a subdifferential lower bound for the iterates gap. Let 
$\{u^k\}$ be generated by Algorithm \ref{alg}, when $k \geq K$, according 
to \cite[Theorem 2.3]{zheng2020globally}, we have the following formula:
\begin{equation}\label{formula}
\partial \left(\lambda \sum_{j \in \Gamma_c} |u^k_j-f_j|^p\right)
=\sum_{j \in \Gamma_c}({\rm ker} \, e_j^\top)^\bot=
\sum_{j \in \Gamma_c}{\rm Im} \, e_j=N_{\mathscr{C}}(u^k).
\end{equation}

\begin{lemma}\label{lemma_subd}
Let $\{u^k\}$ be generated by Algorithm \ref{alg} and $k \geq K$. Then, there 
exists  $\alpha^{k+1} \in \partial \mathcal{F}(u^{k+1})$ such that
\begin{equation}
\|\alpha^{k+1} \| \leq \Lambda \|u^{k+1}-u^{k}\|,
\end{equation}
where $\Lambda \geq 0$ is a constant.
\end{lemma}

\begin{proof}
By noting (\ref{formula}), direct calculation gives that
\begin{equation}\label{subd}
\partial\mathcal{F}(u^{k+1})=\, \lambda 
\sum_{j \in \Gamma} w_j^{k+1} \sign(u_j^{k+1}-f_j)e_j + 
N_{\mathscr{C}}(u^{k+1}) + \partial \left(
\sum_{i \in I_{\tau}}|\tau_i| \|D_iu^{k+1}\|\right).
\end{equation}
Since $u^{k+1}$ solves $\overline{\mathcal{H}}^{\delta}_k(u)$ 
in (\ref{pro_linlast}), by the first order optimality condition, we have
\begin{equation*}
0 \in \partial \overline{\mathcal{H}}^{\delta}_k(u^{k+1}) = 
\partial \overline{\mathcal{H}}_k(u^{k+1})
+ \partial \delta_{\mathscr{C}}(u^{k+1}).
\end{equation*}
Then, we have
\begin{equation}\label{27}
-\lambda 
\sum_{j \in \Gamma} w_j^{k} \sign(u_j^{k+1}-f_j)e_j - 
\rho(u^{k+1}-u^{k}) \in \partial \left(
\sum_{i \in I_{\tau}}|\tau_i| \|D_iu^{k+1}\|\right) + 
N_{\mathscr{C}}(u^{k+1}).
\end{equation}
Therefore, combining (\ref{subd}) and (\ref{27}) yields that
$$
\alpha^{k+1}:=\lambda 
\sum_{j \in \Gamma} (w_j^{k+1}-w_j^k) \sign(u_j^{k+1}-f_j)e_j 
- \rho(u^{k+1}-u^{k}) \in \partial\mathcal{F}(u^{k+1}).
$$

By Algorithm \ref{alg}, we have $|u_j^k-f_j| \geq \epsilon$, $\forall j \in 
\Gamma, \forall k \geq K$. Consider the function $x \mapsto 
x^{p-1}$. For any $x_2 > x_1 \geq \epsilon,$ by the mean value 
theorem, there exists $x_0 \in (x_1, x_2)$ such that
$$
|x_2^{p-1}-x_1^{p-1}|=(1-p)x_0^{p-2}|x_2-x_1| \leq (1-p) 
\epsilon^{p-2}|x_2-x_1|.
$$ 
Then, by the definition (\ref{sigwei}), it follows that:

\begin{equation}\label{lipsch}
\begin{aligned}
|w_j^{k+1}-w_j^{k}| &= p\big||u_j^{k+1}-f_j|^{p-1}-|u_j^k-f_j|^{p-1}\big| \\ 
& \leq  p(1-p)\epsilon^{p-2} |u_j^{k+1}-u_j^{k} | \leq p(1-p)\epsilon^{p-2} 
\left\| u^{k+1}-u^{k} \right\|.
\end{aligned}
\end{equation}
Therefore,
\begin{equation*}
\begin{aligned}
\|\alpha^{k+1} \| & \leq  \lambda \sum_{j \in \Gamma}  |w_j^{k+1}-w_j^{k}| 
+ \rho \|u^{k+1}-u^k \| \\
& \overset{(\ref{lipsch})}{\leq} \left(\lambda p(1-p)\epsilon^{p-2} 
N_v + \rho \right) \|u^{k+1}-u^k \|  \\
& =: \Lambda \|u^{k+1}-u^{k}\|,
\end{aligned}
\end{equation*}
where $\Lambda := \lambda p(1-p)\epsilon^{p-2} N_v + \rho$.
\end{proof}

\begin{theorem}
Let $\{u^k\}$ be generated by Algorithm \ref{alg}. Then $\{u^k \}$ converges 
to a critical point $u^*$ of $\mathcal{F}(u)$.
\end{theorem}
\begin{proof}
Since $\{u^k\}$ is bounded (see Lemma \ref{lemma_decrease}) and 
$\mathcal{F}(u)$ is continuous, there exists a subsequence $\{u^{k_j}\}$ 
and $\hat{u}$ such that 
\begin{equation}
u^{k_j} \rightarrow \hat{u} \, \mbox{ and } \, \mathcal{F}(u^{k_j}) 
\rightarrow \mathcal{F}(\hat{u}), \quad \mbox{as }  j \rightarrow +\infty. 
\end{equation}
Combining (\ref{lemma_decrease}) and (\ref{lemma_subd}), by 
\cite[Theorem 2.9]{attouch2013convergence}, the sequence $\{u^k\}$
converges to $u^* = \hat{u}$ globally, and $u^*$ is a critical 
point of $\mathcal{F}(u)$.
\end{proof}

\section{Experiments}\label{sec5}

In this section, we provide some numerical examples. The proposed 
algorithms are implemented with C++. All experiments are performed on 
Visual Studio 2019 on a desktop computer
(Intel Core i7-8700 CPU @ 3.20GHz, 16G RAM). The peak signal-to-noise 
ratio (PSNR) is used to measure the quality of restored 
images. For grayscale images, the PSNR is defined as follows:
$$
\mbox{PSNR} = 10*\log_{10}\frac{\mathrm{N}_v}{\|u-\underline{u}\|^2}\mbox{dB},
$$
where $\mathrm{N}_v$ is the number of the vertices,  $\underline{u}$ is 
the original image, and $u$ is the restored image. The PSNR of color images 
is defined similarly.

The stopping tolerance of PLM is set as
$$
\frac{\|u^{k+1}-u^{k}\|}{\|u^{k+1}\|}< 10^{-6} \mbox{ or }
\frac{\|\bm{u}^{k+1}-\bm{u}^{k}\|}{\|\bm{u}^{k+1}\|}< 10^{-6}.
$$
The maximum iteration number is set to be 500. Generally, the choice 
of the initial point is crucial in nonconvex optimization. In this 
specific problem, if the initialization is distant from the original 
image, the support of the initial point may significantly differ from 
that of the original image. Considering that L$_1$TV model can be 
viewed as a reliable convex approximation of L$_p$TV model, the solution 
obtained from the L$_1$TV model is likely to remain close to the original 
image. Motivated by this observation, we use the result of the L$_1$TV 
model, solved by ADMM \cite{wu2012augmented}, as the initialization 
$u^0$ or $\bm{u}^0$.

\begin{figure}[!ht]
\begin{center}
\begin{tabular}{ccc}
\includegraphics[height=2.2cm]{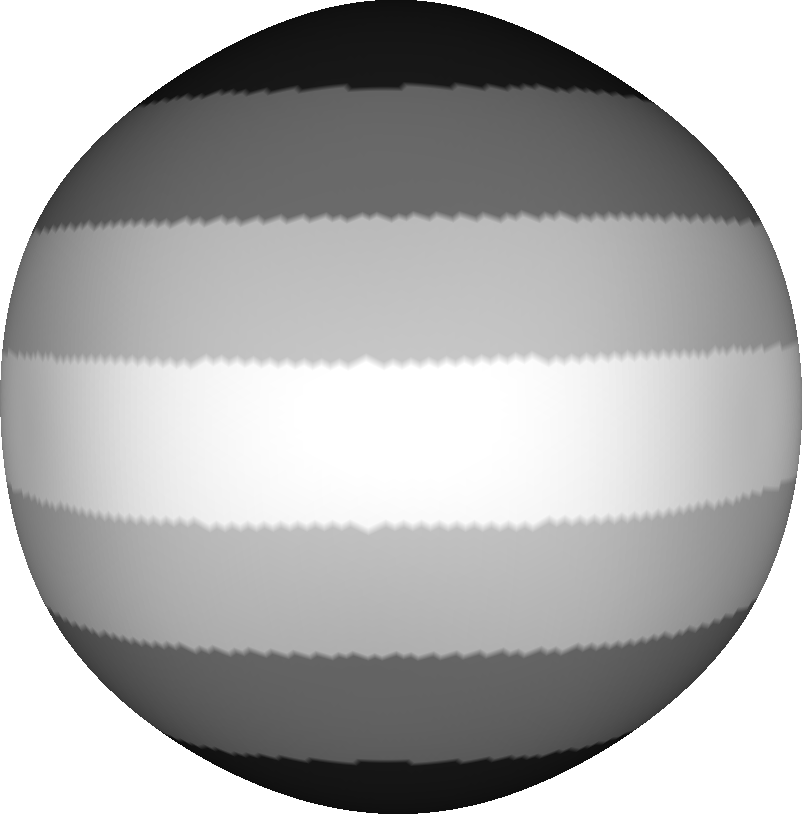}  &
\includegraphics[height=3.0cm]{bottle.png}  &
\includegraphics[height=2.6cm]{bunny.png} \\
\includegraphics[height=2.2cm]{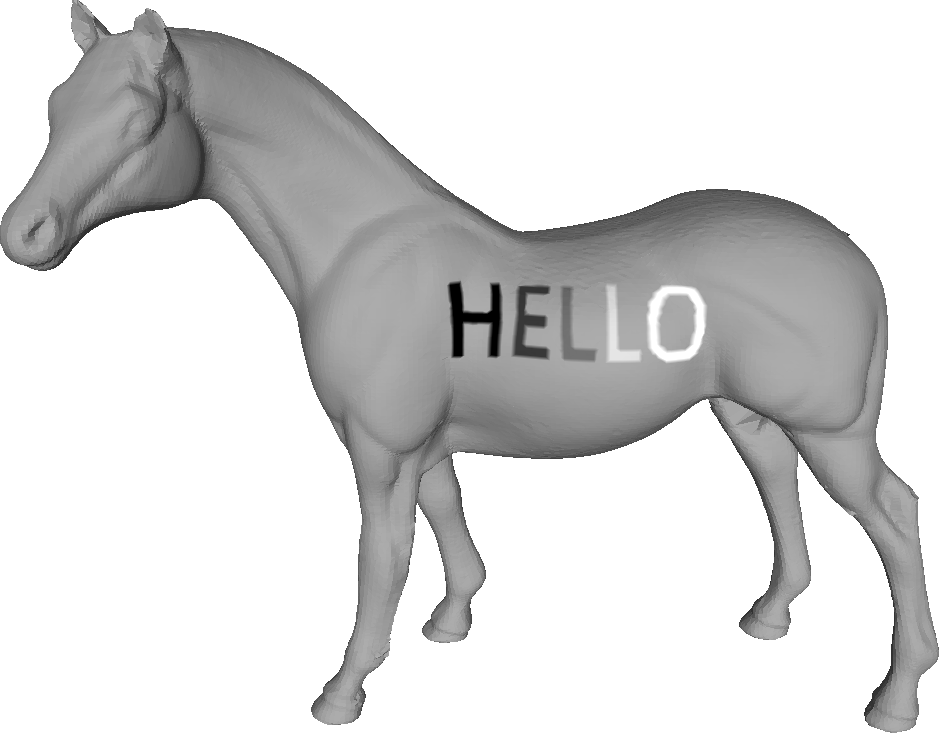}  &
\includegraphics[height=1.7cm]{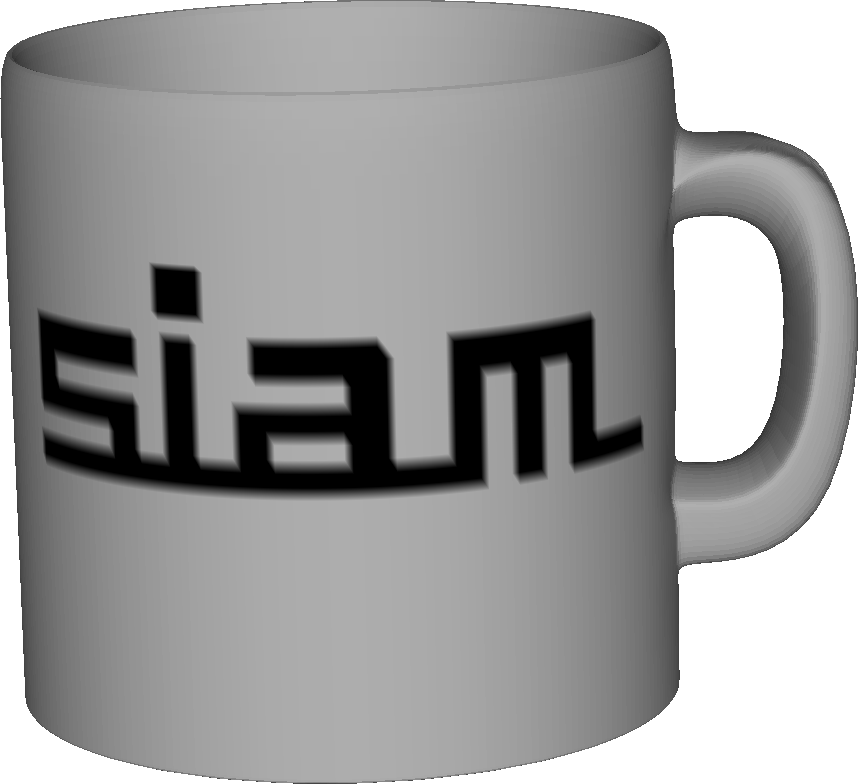}  &
\includegraphics[height=2.2cm]{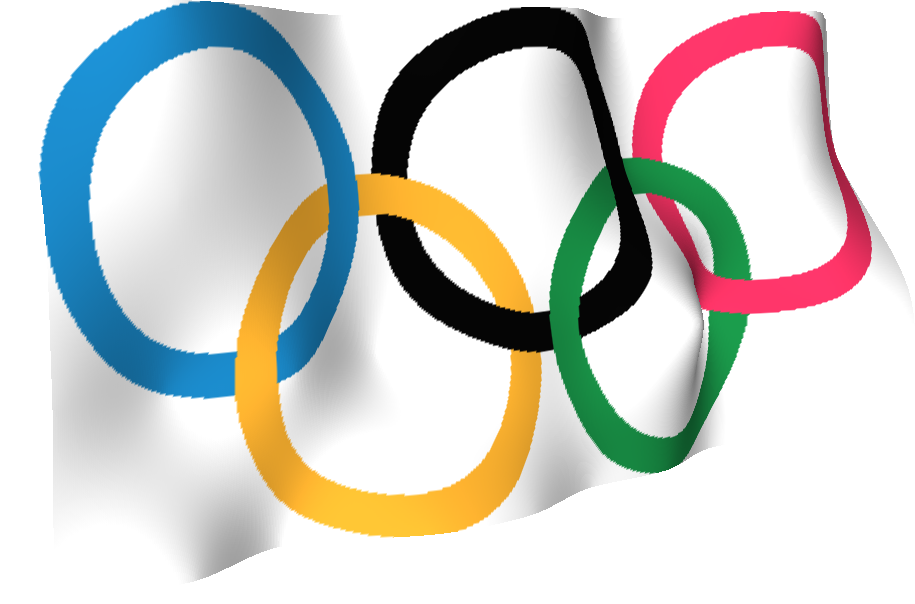}  
\end{tabular}
\caption{Test images. From top left to bottom right: Ball, Bottle, Bunny,
Horse, Mug and Flag.}
\label{testimage}
\end{center}
\end{figure}

We test six images shown in Figure \ref{testimage}, including four grayscale 
images and two color images. The sizes of the underlying meshes are shown 
in Table \ref{infimage}. Suppose the original grayscale image is 
$\underline{u}$. We impose some salt-and-pepper noise and obtain
the observed image by
$$
f = \underline{u} + n,
$$
The noise of color image is imposed similarly.

\begin{table}[!ht]
\footnotesize
  \centering
  \caption{Underlying meshes of test images}
  \label{infimage}
  {\renewcommand{\arraystretch}{1.1}
  \begin{tabular}{|c|cc|}
    \hline
    Image  &  Number of vertices & Number of triangles  \\
    \hline
    \hline
    Ball  & 40962 &  81920 \\
    Bottle  & 132300 & 264600 \\
    Horse  & 48480 & 96956 \\
    Mug  & 29664 & 59328 \\
    \hline
    Bunny  & 34817 & 69630 \\
    Flag  & 163116 &  324500 \\
    \hline
    \end{tabular}
  }
\end{table}

\subsection{Performance of PLM}

In this subsection, we show the performance of the proposed method
and compare it with the L$_1$TV model solved by ADMM. We consider four 
levels of noise $\rho = 0.05,0.10,0.20,0.30$ and $p=0.1,0.3,0.5,0.7,0.9$.
The noisy images with $\rho=0.10$ are shown in Figure \ref{noise}.

\begin{figure}[!ht]
\begin{center}
\begin{tabular}{cccc}
\includegraphics[height=2.2cm]{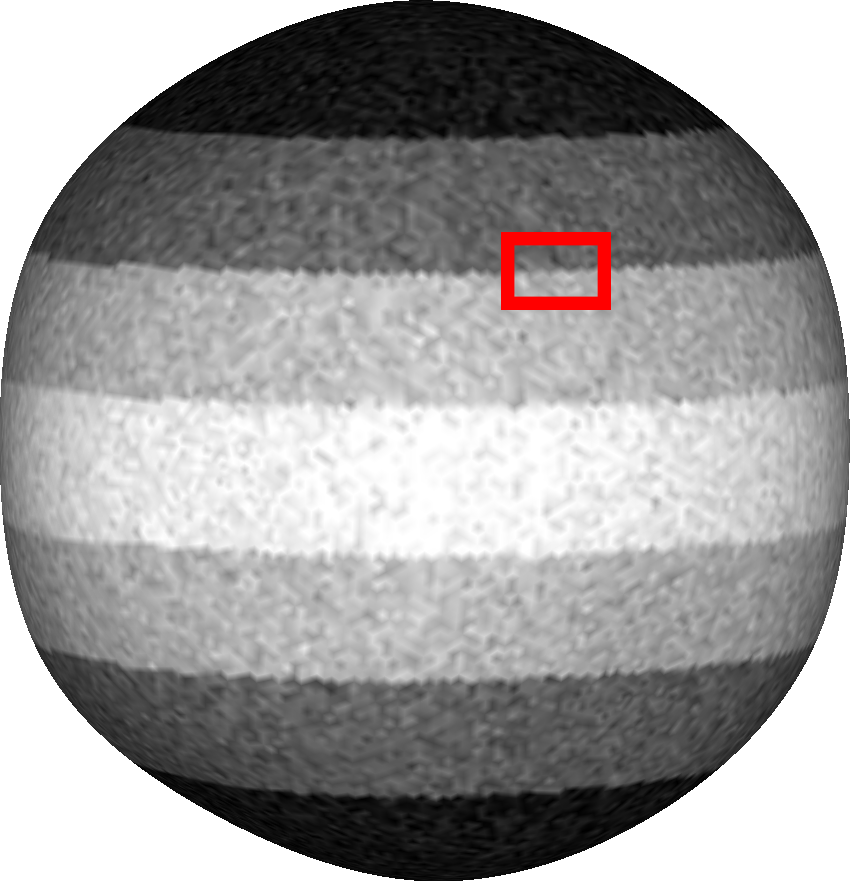} &
\includegraphics[width=1.7cm]{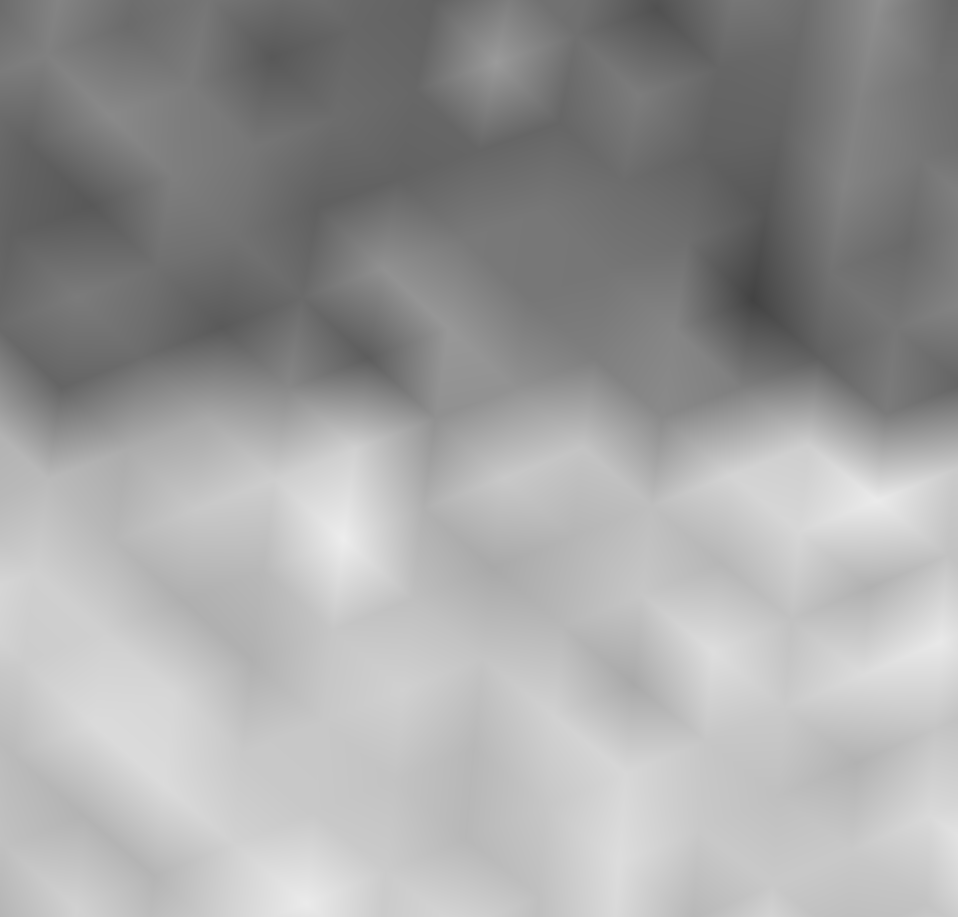}   &
\includegraphics[height=2.8cm]{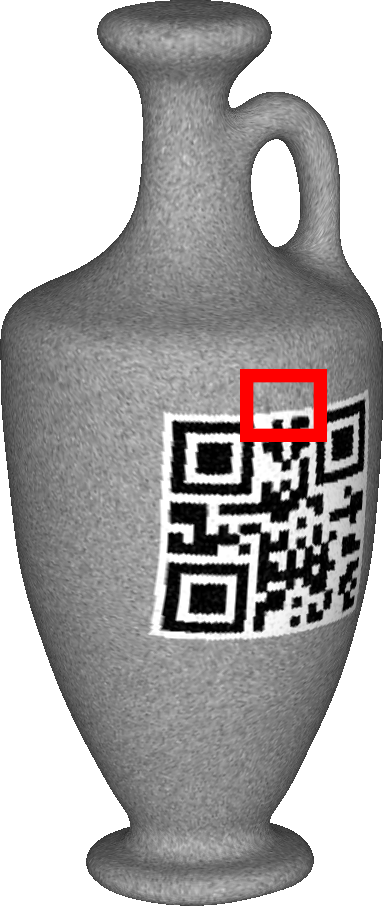} &
\includegraphics[width=1.7cm]{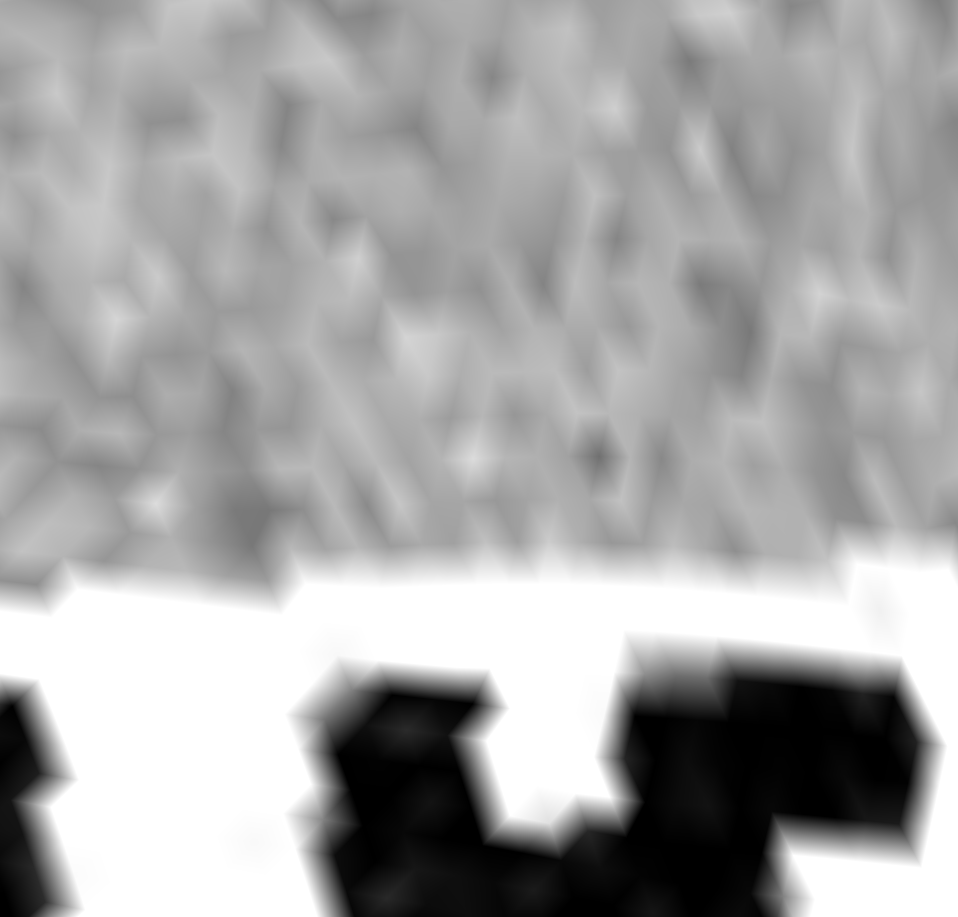}  \\[1mm]
\includegraphics[height=2.5cm]{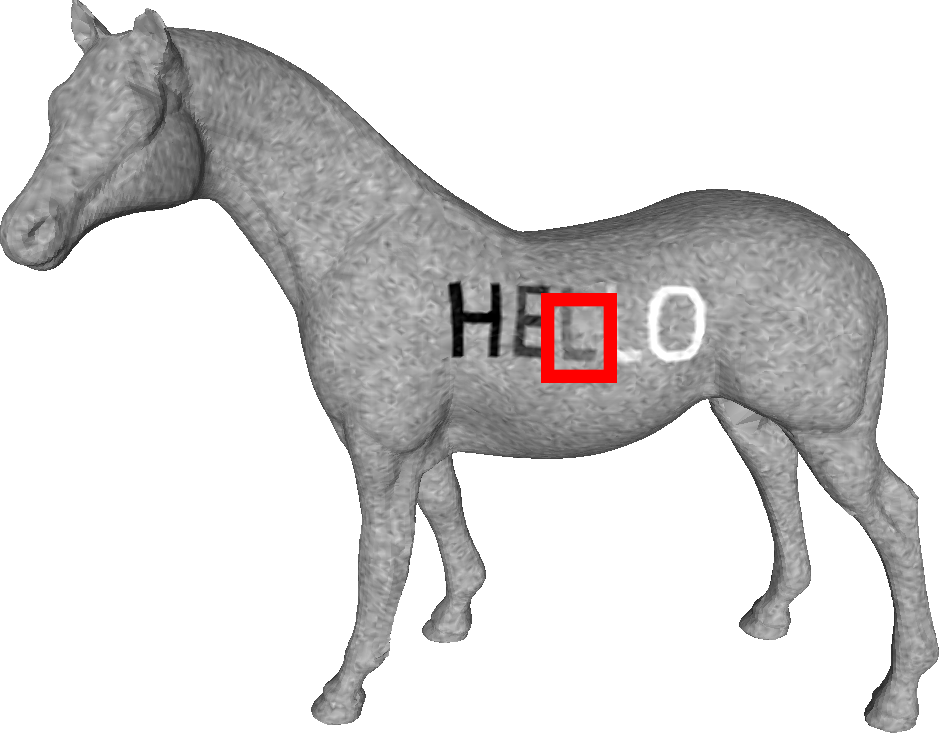}  &
\includegraphics[width=1.7cm]{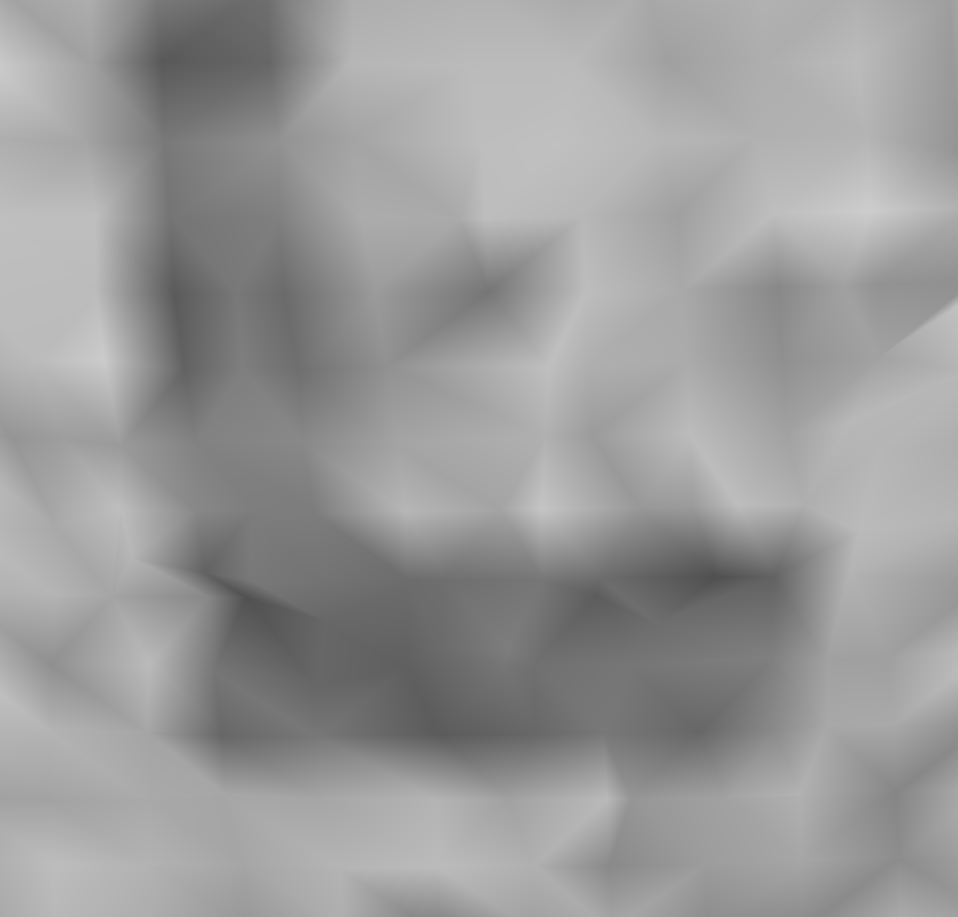}   &
\includegraphics[height=2.3cm]{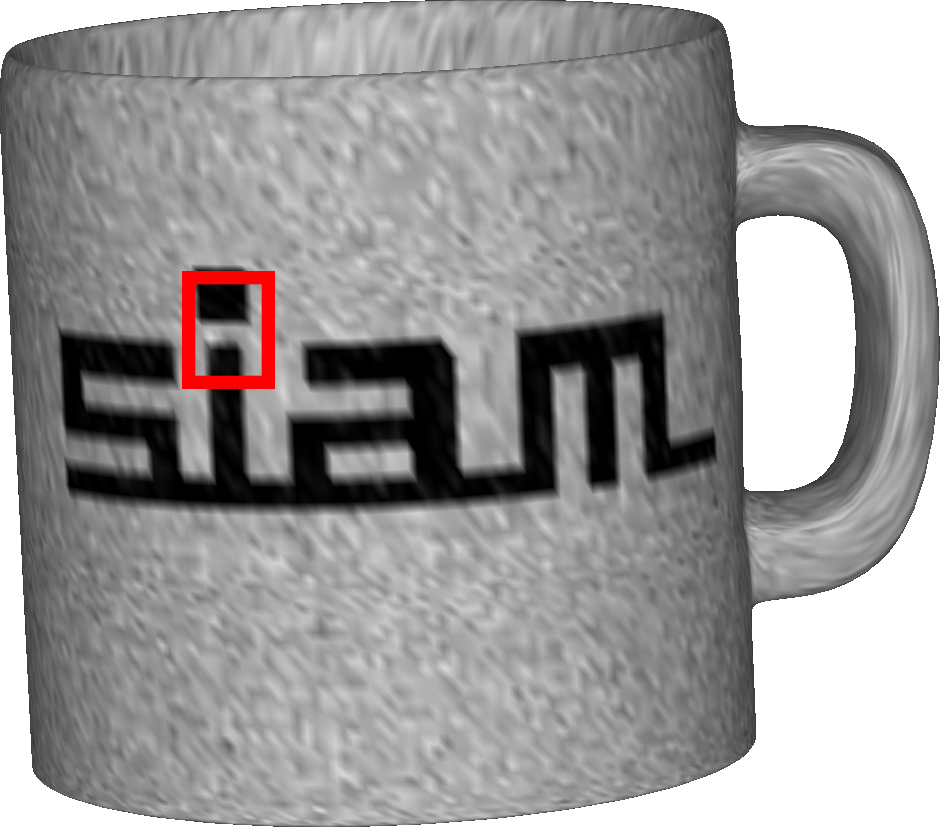}  &
\includegraphics[width=1.7cm]{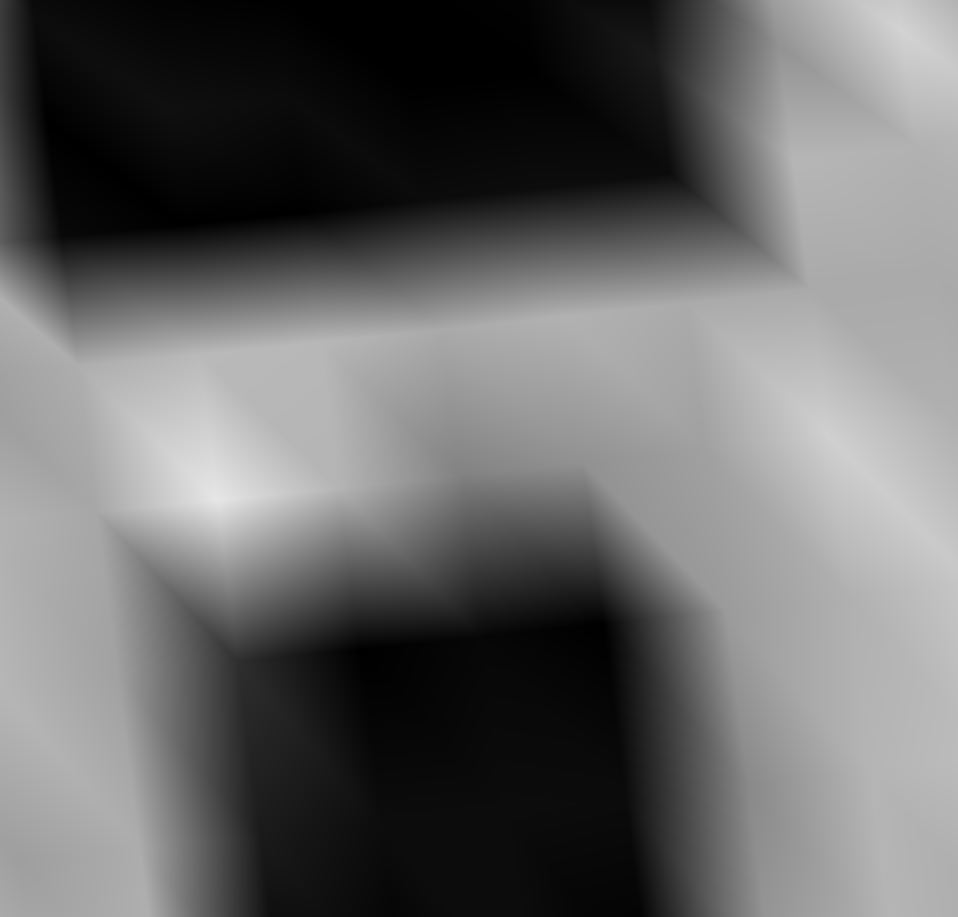}   \\[1mm]
\includegraphics[height=2.5cm]{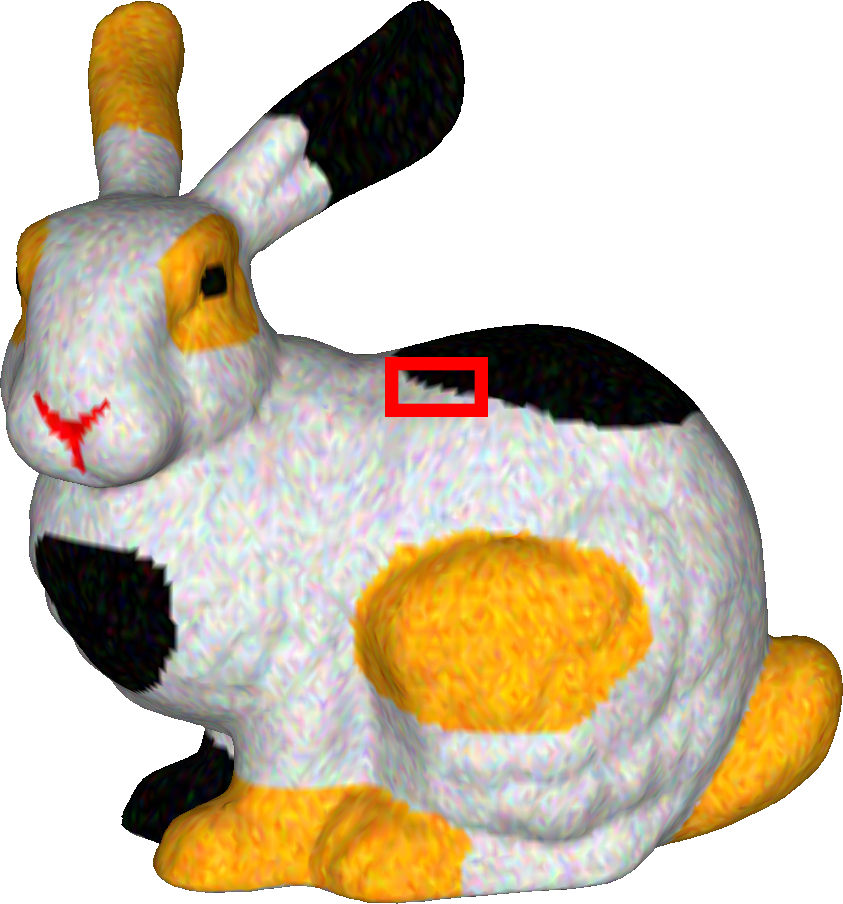}  &
\includegraphics[width=1.7cm]{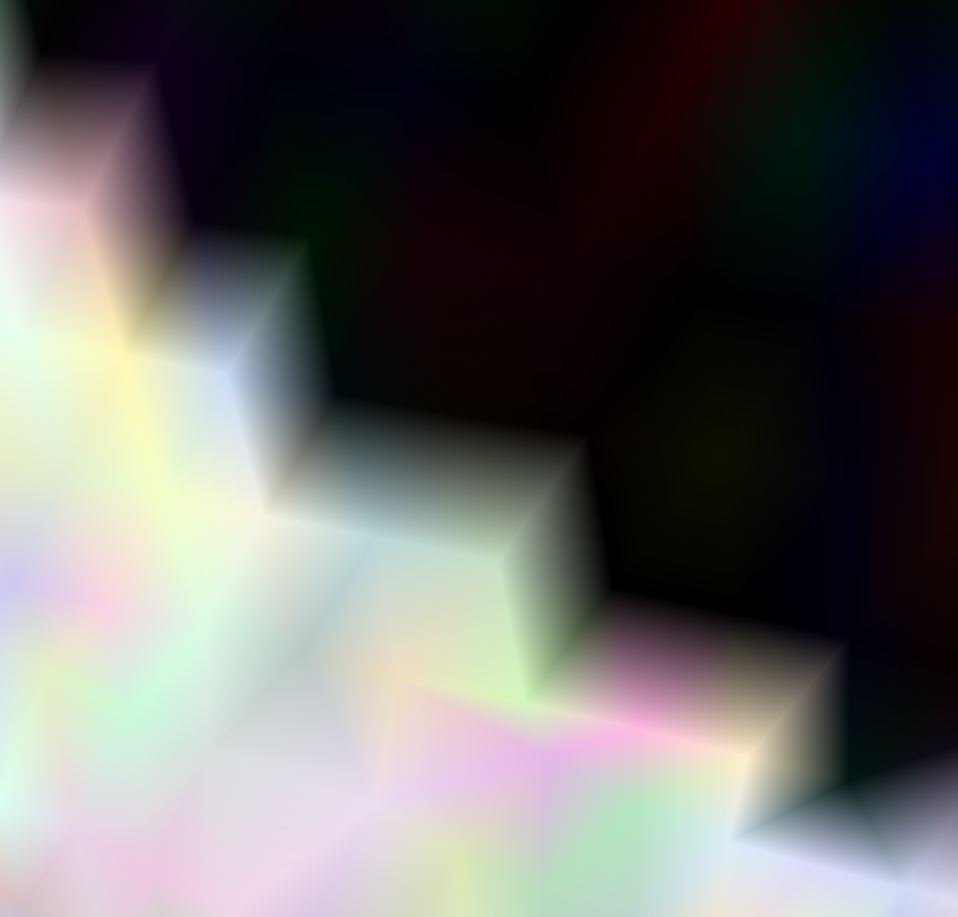}   &
\includegraphics[height=2.3cm]{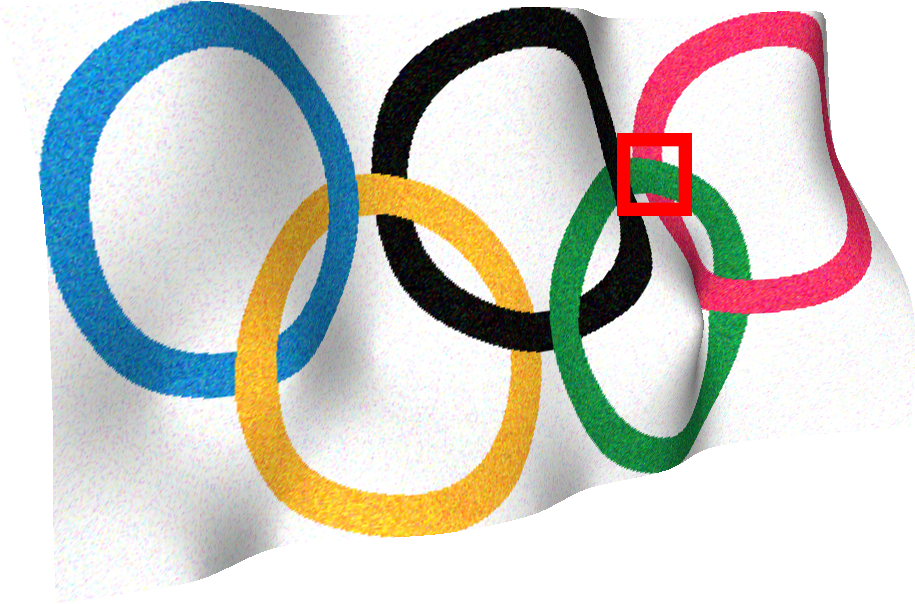}  &
\includegraphics[width=1.7cm]{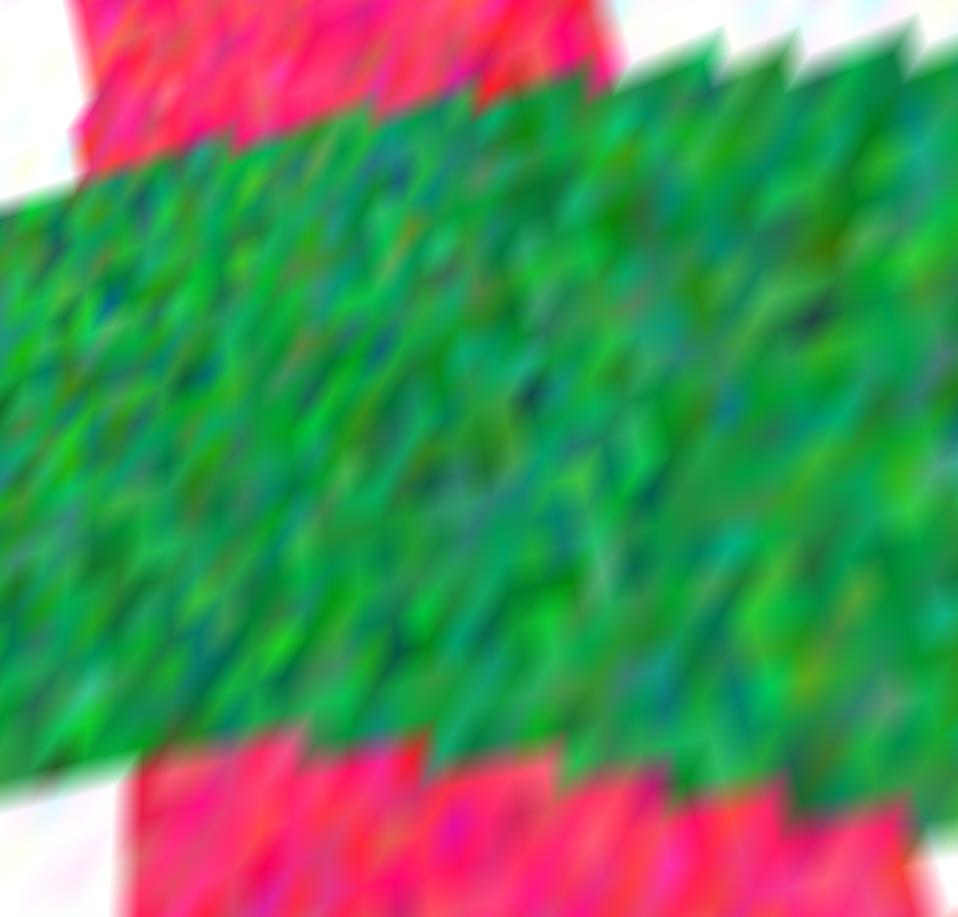}  
\end{tabular}
\caption{Noisy images with zoomed regions, where the noise level is 
$\rho=0.10$. The second (resp., fourth) column is the corresponding 
zoomed regions of the first (resp., third) column.}
\label{noise}
\end{center}
\end{figure}

\begin{table}[!ht]
\footnotesize
  \centering
  \caption{Comparison results of different methods. The running time 
  is measured in seconds. The best results among L$_p$TV are highlighted
  in boldface.}
  \label{table_de}
  {\renewcommand{\arraystretch}{1.1}
  \begin{tabular}{|c|c|c|c|ccccc|}
    \hline
    \multirow{2}{*}{Image}  & Noise  & & \multirow{2}{*}{L$_1$TV}
    & \multicolumn{5}{c|}{L$_p$TV}     \\
    \cline{5-9}
    &  level &  &  &  $p=0.9$ & $p=0.7$  & $p=0.5$ &  $p=0.3$ & $p=0.1$ \\
    \hline
    \hline
    \multirow{6}{*}{Ball} & \multirow{2}{*}{0.05} & PSNR & - 
    & - & - & - & - & \textbf{-}  \\
    & & Time  &  -  & - & - & - & - & \textbf{-}    \\
    \cline{2-9}
    & \multirow{2}{*}{0.10} & PSNR & - 
    & -  & -  & - & - & \textbf{-}  \\
    & & Time  & -  & - & - & - & \textbf{-} & -  \\
    \cline{2-9}
    & \multirow{2}{*}{0.20} & PSNR & - 
    & -  & - & -  & - & \textbf{-}   \\
    & & Time  & - & - & - & - & \textbf{-} & -   \\
    \cline{2-9}
    & \multirow{2}{*}{0.30} & PSNR & - 
    & -  & -  & - & - & \textbf{-}  \\
    & & Time  & -  & - & - & - & \textbf{-} & -  \\
     \hline
    \multirow{6}{*}{Bottle} & \multirow{2}{*}{0.05} & PSNR & - 
    & - & - & - & - & \textbf{-}  \\
    & & Time  &  -  & - & - & - & - & \textbf{-}    \\
    \cline{2-9}
    & \multirow{2}{*}{0.10} & PSNR & - 
    & -  & -  & - & - & \textbf{-}  \\
    & & Time  & -  & - & - & - & \textbf{-} & -  \\
    \cline{2-9}
    & \multirow{2}{*}{0.20} & PSNR & - 
    & -  & - & -  & - & \textbf{-}   \\
    & & Time  & - & - & - & - & \textbf{-} & -   \\
    \cline{2-9}
    & \multirow{2}{*}{0.30} & PSNR & - 
    & -  & -  & - & - & \textbf{-}  \\
    & & Time  & -  & - & - & - & \textbf{-} & -  \\
     \hline
    \multirow{6}{*}{Horse} & \multirow{2}{*}{0.05} & PSNR & - 
    & - & - & - & - & \textbf{-}  \\
    & & Time  &  -  & - & - & - & - & \textbf{-}    \\
    \cline{2-9}
    & \multirow{2}{*}{0.10} & PSNR & - 
    & -  & -  & - & - & \textbf{-}  \\
    & & Time  & -  & - & - & - & \textbf{-} & -  \\
    \cline{2-9}
    & \multirow{2}{*}{0.20} & PSNR & - 
    & -  & - & -  & - & \textbf{-}   \\
    & & Time  & - & - & - & - & \textbf{-} & -   \\
    \cline{2-9}
    & \multirow{2}{*}{0.30} & PSNR & - 
    & -  & -  & - & - & \textbf{-}  \\
    & & Time  & -  & - & - & - & \textbf{-} & -  \\
     \hline
    \multirow{6}{*}{Mug} & \multirow{2}{*}{0.05} & PSNR & 35.61 
    & 36.90 & 38.75 & 39.39 & 39.65 & \textbf{39.71}  \\
    & & Time  &  16.8  & 36.6 & 32.9 & 24.5 & 23.1 & \textbf{21.6}    \\
    \cline{2-9}
    & \multirow{2}{*}{0.10} & PSNR & 34.21 
    & 34.69  & 36.15  & 36.65 & 37.01 & \textbf{37.19}  \\
    & & Time  & 16.4 & 34.5 & 38.3 & 28.6 & \textbf{26.0} & 32.68  \\
    \cline{2-9}
    & \multirow{2}{*}{0.20} & PSNR & 31.77 
    & 32.42  & 33.77 & 34.32  & 34.53 & \textbf{34.61}   \\
    & & Time  & 18.6 & 247.7 & 36.7 & 50.5 & 37.2 & \textbf{36.0}   \\ 
    \cline{2-9}
    & \multirow{2}{*}{0.30} & PSNR & 30.46 
    & 31.10  & 31.95  & 32.41 & 32.68 & \textbf{32.74}  \\
    & & Time  & 20.9  & 44.4 & 67.3 & 47.0 & 50.2 &\textbf{39.5}  \\
     \hline
    \multirow{6}{*}{Bunny} & \multirow{2}{*}{0.05} & PSNR & - 
    & - & - & - & - & \textbf{-}  \\
    & & Time  &  -  & - & - & - & - & \textbf{-}    \\
    \cline{2-9}
    & \multirow{2}{*}{0.10} & PSNR & - 
    & -  & -  & - & - & \textbf{-}  \\
    & & Time  & -  & - & - & - & \textbf{-} & -  \\
    \cline{2-9}
    & \multirow{2}{*}{0.20} & PSNR & - 
    & -  & - & -  & - & \textbf{-}   \\
    & & Time  & - & - & - & - & \textbf{-} & -   \\
    \cline{2-9}
    & \multirow{2}{*}{0.30} & PSNR & - 
    & -  & -  & - & - & \textbf{-}  \\
    & & Time  & -  & - & - & - & \textbf{-} & -  \\
     \hline
    \multirow{6}{*}{Flag} & \multirow{2}{*}{0.05} & PSNR & - 
    & - & - & - & - & \textbf{-}  \\
    & & Time  &  -  & - & - & - & - & \textbf{-}    \\
    \cline{2-9}
    & \multirow{2}{*}{0.10} & PSNR & - 
    & -  & -  & - & - & \textbf{-}  \\
    & & Time  & -  & - & - & - & \textbf{-} & -  \\
    \cline{2-9}
    & \multirow{2}{*}{0.20} & PSNR & - 
    & -  & - & -  & - & \textbf{-}   \\
    & & Time  & - & - & - & - & \textbf{-} & -   \\
    \cline{2-9}
    & \multirow{2}{*}{0.30} & PSNR & - 
    & -  & -  & - & - & \textbf{-}  \\
    & & Time  & -  & - & - & - & \textbf{-} & -  \\
     \hline
    \end{tabular}
  }
\end{table}

The comparison results are shown in Table \ref{table_de}, where all results 
are averaged over 10 trials for each case. We consider the PSNR. L$_p$TV 
outperforms L$_1$TV for all cases. For most cases, the smaller $p$ is, the higher 
PSNR is obtained. This is very natural, because a smaller $p$ would enhance the 
sparsity of the data fitting term of the solution. 
Since we consider 
recovering images corrupted by salt-and-pepper noise, L$_p$TV with a 
smaller $p$ can yield better denoising results. In addition, it is 
observed that as the noise level increases, the performance of the 
L$_p$TV model degrades. This decline is attributed to the fact that, 
at higher noise levels, the requirement for sparse reconstruction in 
the data fitting term diminishes.

The visual comparison of some restorations is shown in Figure \ref{recover1} 
and \ref{recover2}. All results are with respect to $\rho=0.10$. From the 
zoomed regions, we can see the recovered results by the L$_p$ TV model,
especially by $p=0.1$, have much better performances than those by the L$_1$TV model.

\section{Conclusions}\label{sec6}

In this paper, we consider recovering images corrupted by salt-and-pepper 
noise on surfaces with nonconvex variational methods, i.e., proposing 
the L$_p$TV model on surfaces. From the perspective of sparse 
reconstruction, L$_p$TV model is suitable for salt-and-pepper noise 
removal. The lower bound property of the model is presented. Motivated 
by the lower bound property, we adopt proximal linearization method with 
the thresholding and support shrinking strategy. We establish the global 
convergence of the algorithm. Numerical examples are given to show good 
performance of the algorithm.

We adopt linear interpolation in this work. Linear interpolation 
is a basic way for interpolation and leads to
a simple relationship between the image $u$ and its gradient $Du$;
see (\ref{BasisonMesh}). Based on this, we can prove 
the lower bound theory, i.e., Theorem \ref{lowerbound}. Moreover,
the simple relationship between $u$ and $Du$ brings convenience
for computation. If other interpolation methods are used, do we still have 
the lower bound theory and an efficient algorithm? This will be 
investigated in the future. 

There are other image processing problems on surfaces, such as image 
segmentation \cite{benninghoff2016segmentation,delaunoy2009convex,
huska2018convex,wu2012augmented,spira2005segmentation},
image decomposition \cite{wu2013variational,huska2019convex},
and image inpainting \cite{wu2005inpainting}. We will discuss how
to extend this work to these problems in the future. 

\appendix

\section{The subdifferential}\label{apsd}

\begin{definition}(Subdifferentials) Let $f : \mathbb{R}^n \to 
(-\infty, +\infty]$ be a proper, lower semicontinuous function. For a 
point $x \in \text{dom } f$,
\begin{itemize}
    \item[(i)] The regular subdifferential of $f$ at $x$ is defined as
    \[
    \hat{\partial} f(x) = \left\{ u \in \mathbb{R}^n : 
    \liminf_{y \to x, y \neq x} \frac{f(y) - f(x) - \langle u, 
    y - x \rangle}{\| y - x \|_2} \geq 0 \right\};
    \]
    \item[(ii)] The (limiting) subdifferential of $f$ at $x$ is defined as
    \[
    \partial f(x) = \left\{ u \in \mathbb{R}^n : \exists x^k \to x, 
    f(x^k) \to f(x), u^k \in \hat{\partial} f(x^k), u^k \to u \text{ as } 
    k \to \infty \right\}.
    \]
\end{itemize}
\end{definition}

\begin{remark}
For any $x \in \text{dom } f$, $\hat{\partial} f(x) \subset \partial f(x)$. 
If $f$ is differentiable at $x$, then $\hat{\partial} f(x) = \partial f(x) 
= \{ \nabla f(x) \}$.
\end{remark}

\begin{remark}
A point $x$ is a critical point of $f(x)$ if and only if 
$0 \in \partial f(x)$. If $x^{*}$ is a local minimizer of $f(x)$, then 
$0 \in \partial f(x^{*})$.
\end{remark}

\section{KL functions}\label{apKL}

\begin{definition}(KL property and KL functions)
The function $f: \mathbb{R}^n \rightarrow (-\infty,+\infty]$ has 
KL property at $\hat{x} \in \operatorname{dom}\partial f$ if there 
exist $a \in (0,+\infty]$, a neighborhood $U$ of $\hat{x}$, and a 
continuous concave function $\varphi: [0,a) \rightarrow [0,+\infty)$ such that
\begin{enumerate}[(i)]
    \item $\varphi(0) = 0$;
    \item $\varphi$ is $C^1$ on $(0,a)$;
    \item for all $s \in (0,a)$, $\varphi'(s) > 0$;
    \item for all $x \in U \cap \{f(\hat{x}) < f(x) < f(\hat{x}) + a\}$, 
    the Kurdyka-Łojasiewicz inequality holds:
\end{enumerate}
\[
\varphi'(f(x) - f(\hat{x})) \operatorname{dist}(0, \partial f(x)) \geq 1.
\]

If the function $f$ satisfies the KL property at each point of 
$\operatorname{dom}\partial f$, $f$ is called a KL function.
\end{definition}

{\small
\bibliographystyle{abbrv}
\bibliography{ref}
}

\end{document}